\documentclass[12pt]{amsart}
\usepackage{amsmath,amsthm,amsfonts,amssymb,mathrsfs}
\date{\today}

\input xymatrix
\xyoption{all}

\usepackage{color}

\newtheorem{theorem}{Theorem}
\newtheorem{proposition}[theorem]{Proposition}
\newtheorem{corollary}[theorem]{Corollary}
\newtheorem{lemma}[theorem]{Lemma}

\theoremstyle{definition}
\newtheorem{remark}[theorem]{Remark}

\usepackage{hyperref}

\begin{document}

\title[On the monoids of monotone injective partial selfmaps of $\mathbb{N}^{2}_{\leqslant}$ ...]{On the monoid of monotone injective partial selfmaps of $\mathbb{N}^{2}_{\leqslant}$ with cofinite domains and images, II}

\author[O.~Gutik and I.~Pozdniakova]{Oleg~Gutik and Inna Pozdniakova}
\address{Department of Mathematics, National University of Lviv, Universytetska 1, Lviv, 79000, Ukraine}
\email{o\_gutik@franko.lviv.ua, ovgutik@yahoo.com, pozdnyakova.inna@gmail.com}



\begin{abstract}
Let $\mathbb{N}^{2}_{\leqslant}$ be the set $\mathbb{N}^{2}$ with the partial order defined as the product of usual order $\leq$ on the set of positive integers $\mathbb{N}$. We study the semigroup $\mathscr{P\!O}\!_{\infty}(\mathbb{N}^2_{\leqslant})$ of monotone injective partial selfmaps of $\mathbb{N}^{2}_{\leqslant}$ having cofinite domain and image. We describe the natural partial order on the semigroup $\mathscr{P\!O}\!_{\infty}(\mathbb{N}^2_{\leqslant})$ and show that it coincides with the natural partial order which is induced from symmetric inverse monoid $\mathscr{I}_{\mathbb{N}\times\mathbb{N}}$ over the set $\mathbb{N}\times\mathbb{N}$ onto the semigroup $\mathscr{P\!O}\!_{\infty}(\mathbb{N}^2_{\leqslant})$. We proved that the semigroup $\mathscr{P\!O}\!_{\infty}(\mathbb{N}^2_{\leqslant})$ is isomorphic to the semidirect product $\mathscr{P\!O}\!_{\infty}^{\,+}(\mathbb{N}^2_{\leqslant})\rtimes \mathbb{Z}_2$ of the monoid $\mathscr{P\!O}\!_{\infty}^{\,+}(\mathbb{N}^2_{\leqslant})$ of orientation-preserving monotone injective partial selfmaps of $\mathbb{N}^{2}_{\leqslant}$ with cofinite domains and images by the cyclic group $\mathbb{Z}_2$ of the order two. Also we describe the congruence $\sigma$ on the semigroup $\mathscr{P\!O}\!_{\infty}(\mathbb{N}^2_{\leqslant})$ which is generated by the natural order $\preccurlyeq$ on the semigroup $\mathscr{P\!O}\!_{\infty}(\mathbb{N}^2_{\leqslant})$: $\alpha\sigma\beta$ if and only if $\alpha$ and $\beta$ are comparable in $\left(\mathscr{P\!O}\!_{\infty}(\mathbb{N}^2_{\leqslant}),\preccurlyeq\right)$. We prove that the quotient semigroup $\mathscr{P\!O}\!_{\infty}^{\,+}(\mathbb{N}^2_{\leqslant})/\sigma$ is isomorphic to the free commutative monoid $\mathfrak{AM}_\omega$ over an infinite countable set and show that the quotient semigroup $\mathscr{P\!O}\!_{\infty}(\mathbb{N}^2_{\leqslant})/\sigma$ is isomorphic to the semidirect product of the free commutative monoid $\mathfrak{AM}_\omega$ by the group $\mathbb{Z}_2$.
\end{abstract}

\maketitle


We shall follow the terminology of~\cite{Clifford-Preston-1961-1967} and \cite{Howie-1995}.

In this paper we shall denote the first infinite cardinal by $\omega$ and the cardinality of the set $A$ by $|A|$.  We shall identify every set $X$ with its cardinality $|X|$. By $\mathbb{Z}_2$ we shall denote the cyclic group of order two. Also, for infinite subsets $A$ and $B$ of an infinite set $X$ we shall write $A{\subseteq^*}B$ if and only if there exists a finite subset $A_0$ of $A$ such that $A\setminus A_0\subseteq B$.

An algebraic semigroup $S$ is called {\it inverse} if for any element $x\in S$ there exists a unique $x^{-1}\in S$ such that $xx^{-1}x=x$ and $x^{-1}xx^{-1}=x^{-1}$. The element $x^{-1}$ is called the {\it inverse of} $x\in S$.

If $S$ is a semigroup, then we shall denote the subset of idempotents in $S$ by $E(S)$. If $S$ is an inverse semigroup, then $E(S)$ is closed under multiplication and we shall refer to $E(S)$ a \emph{band} (or the \emph{band of} $S$). If the band $E(S)$ is a non-empty subset of $S$, then the semigroup operation on $S$ determines the following partial order $\leqslant$ on $E(S)$: $e\leqslant f$ if and only if $ef=fe=e$. This order is called the {\em natural partial order} on $E(S)$. A \emph{semilattice} is a commutative semigroup of idempotents.

If $\alpha\colon X\rightharpoonup Y$ is a partial map, then by $\operatorname{dom}\alpha$ and $\operatorname{ran}\alpha$ we denote the domain and the range of $\alpha$, respectively.

Let $\mathscr{I}_\lambda$ denote the set of all partial one-to-one transformations of an infinite set $X$ of cardinality $\lambda$ together with the following semigroup operation: $x(\alpha\beta)=(x\alpha)\beta$ if $x\in\operatorname{dom}(\alpha\beta)=\{ y\in\operatorname{dom}\alpha\mid y\alpha\in\operatorname{dom}\beta\}$,  for $\alpha,\beta\in\mathscr{I}_\lambda$. The semigroup $\mathscr{I}_\lambda$ is called the \emph{symmetric inverse semigroup} over the set $X$~(see \cite[Section~1.9]{Clifford-Preston-1961-1967}). The symmetric inverse semigroup was introduced by Vagner~\cite{Vagner-1952} and it plays a major role in the theory of semigroups. An element $\alpha\in\mathscr{I}_\lambda$ is called \emph{cofinite}, if the sets $\lambda\setminus\operatorname{dom}\alpha$ and $\lambda\setminus\operatorname{ran}\alpha$ are finite.

Let $(X,\leqslant)$ be a partially ordered set (a poset). For an arbitrary $x\in X$ we denote
 \begin{equation*}
{\uparrow}x=\left\{y\in X\colon x\leqslant y\right\}.
\end{equation*}
We shall say that a partial map $\alpha\colon X\rightharpoonup X$ is \emph{monotone} if $x\leqslant y$ implies $(x)\alpha\leqslant(y)\alpha$ for $x,y\in \operatorname{dom}\alpha$.

Let $\mathbb{N}$ be the set of positive integers with the usual linear order $\le$. On the Cartesian product $\mathbb{N}\times\mathbb{N}$ we define the product partial order, i.e.,
\begin{equation*}
    (i,m)\leqslant(j,n) \qquad \hbox{if and only if} \qquad (i\leq j) \quad \hbox{and} \quad (m\leq n).
\end{equation*}
Later the set $\mathbb{N}\times\mathbb{N}$ with so defined partial order will be denoted by $\mathbb{N}^2_{\leqslant}$.

By $\mathscr{P\!O}\!_{\infty}(\mathbb{N}^2_{\leqslant})$ we denote the semigroup of injective partial monotone selfmaps of $\mathbb{N}^2_{\leqslant}$ with cofinite domains and images. Obviously, $\mathscr{P\!O}\!_{\infty}(\mathbb{N}^2_{\leqslant})$ is a submonoid of the symmetric inverse semigroup $\mathscr{I}_\omega$ and $\mathscr{P\!O}\!_{\infty}(\mathbb{N}^2_{\leqslant})$ is a countable semigroup.

Furthermore, we shall denote the identity of the semigroup $\mathscr{P\!O}\!_{\infty}(\mathbb{N}^2_{\leqslant})$ by $\mathbb{I}$ and the group of units of
$\mathscr{P\!O}\!_{\infty}(\mathbb{N}^2_{\leqslant})$ by $H(\mathbb{I})$.

For any positive integer $n$ and an arbitrary $\alpha\in \mathscr{P\!O}\!_{\infty}(\mathbb{N}^2_{\leqslant})$ we denote:
\begin{equation*}
\begin{split}
  \textsf{V}^n =\{(n,j)\colon j\in\mathbb{N}\}; \qquad & \qquad \textsf{H}^n =\{(j,n)\colon j\in\mathbb{N}\};\\
  \textsf{V}^n_{\operatorname{dom}\alpha}=\textsf{V}^n\cap\operatorname{dom}\alpha;   \qquad & \qquad
  \textsf{V}^n_{\operatorname{ran}\alpha}=\textsf{V}^n\cap\operatorname{ran}\alpha;\\
  \textsf{H}^n_{\operatorname{dom}\alpha}=\textsf{H}^n\cap\operatorname{dom}\alpha;   \qquad & \qquad
  \textsf{H}^n_{\operatorname{ran}\alpha}=\textsf{H}^n\cap\operatorname{ran}\alpha,
\end{split}
\end{equation*}
and
\begin{equation*}
  (i_{\alpha[*,j]},j_{\alpha[i,*]})=(i,j)\alpha, \qquad \hbox{for every} \quad (i,j)\in \operatorname{dom}\alpha.
\end{equation*}

\medskip

It well known that each partial injective cofinite selfmap $f$ of $\lambda$ induces a homeomorphism $f^*\colon\lambda^*\rightarrow\lambda^*$ of the remainder $\lambda^*=\beta\lambda\setminus\lambda$ of the Stone-\v{C}ech compactification of the discrete space $\lambda$. Moreover, under some set theoretic axioms (like \textbf{PFA} or \textbf{OCA}), each homeomorphism of $\omega^*$ is induced by some partial injective cofinite selfmap of $\omega$ (see  \cite{ShelahSteprans1989}--\cite{Velickovic1993}). So, the inverse semigroup  $\mathscr{I}^{\mathrm{cf}}_\lambda$ of injective partial selfmaps of an infinite cardinal $\lambda$ with
cofinite domains and images admits a natural homomorphism  $\mathfrak{h}\colon \mathscr{I}^{\mathrm{cf}}_\lambda\rightarrow \mathscr{H}(\lambda^*)$ to the homeomorphism group $\mathscr{H}(\lambda^*)$ of $\lambda^*$ and this homomorphism is surjective under certain set theoretic assumptions.

In the paper \cite{Gutik-Repovs-2015}  algebraic properties of the semigroup
$\mathscr{I}^{\mathrm{cf}}_\lambda$ are studied. It is showed that
$\mathscr{I}^{\mathrm{cf}}_\lambda$ is a bisimple inverse semigroup
and that for every non-empty chain $L$ in
$E(\mathscr{I}^{\mathrm{cf}}_\lambda)$ there exists an inverse
subsemigroup $S$ of $\mathscr{I}^{\mathrm{cf}}_\lambda$ such that
$S$ is isomorphic to the bicyclic semigroup and $L\subseteq E(S)$,
the Green relations on $\mathscr{I}^{\mathrm{cf}}_\lambda$ are described
and it is proved that every non-trivial congruence on
$\mathscr{I}^{\mathrm{cf}}_\lambda$ is a group congruence. Also, the structure of the quotient semigroup $\mathscr{I}^{\mathrm{cf}}_\lambda/\sigma$ is described, where $\sigma$ is the least group congruence on $\mathscr{I}^{\mathrm{cf}}_\lambda$.

The semigroups $\mathscr{I}_{\infty}^{\!\nearrow}(\mathbb{N})$ and $\mathscr{I}_{\infty}^{\!\nearrow}(\mathbb{Z})$ of injective isotone partial selfmaps with cofinite domains and images of positive integers and integers are studied in \cite{Gutik-Repovs-2011} and \cite{Gutik-Repovs-2012}, respectively. It was proved that the semigroups $\mathscr{I}_{\infty}^{\!\nearrow}(\mathbb{N})$ and $\mathscr{I}_{\infty}^{\!\nearrow}(\mathbb{Z})$ have similar properties to the bicyclic semigroup: they are bisimple and  every non-trivial homomorphic image $\mathscr{I}_{\infty}^{\!\nearrow}(\mathbb{N})$ and $\mathscr{I}_{\infty}^{\!\nearrow}(\mathbb{Z})$ is a group, and moreover the semigroup $\mathscr{I}_{\infty}^{\!\nearrow}(\mathbb{N})$ has $\mathbb{Z}(+)$ as a maximal group image and $\mathscr{I}_{\infty}^{\!\nearrow}(\mathbb{Z})$ has $\mathbb{Z}(+)\times\mathbb{Z}(+)$, respectively.

In the paper \cite{Gutik-Pozdnyakova-2014} we studied the semigroup
$\mathscr{I\!O}\!_{\infty}(\mathbb{Z}^n_{\operatorname{lex}})$ of monotone injective partial selfmaps of the set of $L_n\times_{\operatorname{lex}}\mathbb{Z}$ having cofinite domain and image, where $L_n\times_{\operatorname{lex}}\mathbb{Z}$ is the lexicographic product of $n$-elements chain and the set of integers with the usual linear order. In this paper we described
Green's relations on $\mathscr{I\!O}\!_{\infty}(\mathbb{Z}^n_{\operatorname{lex}})$,
showed that the semigroup $\mathscr{I\!O}\!_{\infty}(\mathbb{Z}^n_{\operatorname{lex}})$ is
bisimple and established its projective congruences. Also, we proved that $\mathscr{I\!O}\!_{\infty}(\mathbb{Z}^n_{\operatorname{lex}})$ is finitely generated, every automorphism of $\mathscr{I\!O}\!_{\infty}(\mathbb{Z})$ is inner and showed that in the case $n\geqslant 2$ the semigroup $\mathscr{I\!O}\!_{\infty}(\mathbb{Z}^n_{\operatorname{lex}})$ has non-inner automorphisms. In \cite{Gutik-Pozdnyakova-2014} we also proved that for every positive integer $n$ the quotient semigroup
$\mathscr{I\!O}\!_{\infty}(\mathbb{Z}^n_{\operatorname{lex}})/\sigma$, where $\sigma$ is a least group congruence on $\mathscr{I\!O}\!_{\infty}(\mathbb{Z}^n_{\operatorname{lex}})$, is isomorphic to the direct power $\left(\mathbb{Z}(+)\right)^{2n}$. The structure of the sublattice of congruences on $\mathscr{I\!O}\!_{\infty}(\mathbb{Z}^n_{\operatorname{lex}})$ that are contained in the least group congruence is described in \cite{Gutik-Pozdniakova-2014}.

In the paper \cite{Gutik-Pozdniakova-2016??} we studied algebraic properties of the semigroup $\mathscr{P\!O}\!_{\infty}(\mathbb{N}^2_{\leqslant})$. We described properties of elements of the semigroup $\mathscr{P\!O}\!_{\infty}(\mathbb{N}^2_{\leqslant})$ as monotone partial bijection of $\mathbb{N}^{2}_{\leqslant}$ and showed that the group of units of $\mathscr{P\!O}\!_{\infty}(\mathbb{N}^2_{\leqslant})$ is isomorphic to the cyclic group of order two. Also in \cite{Gutik-Pozdniakova-2016??} the subsemigroup of idempotents of $\mathscr{P\!O}\!_{\infty}(\mathbb{N}^2_{\leqslant})$ and the Green relations on $\mathscr{P\!O}\!_{\infty}(\mathbb{N}^2_{\leqslant})$ are described. In particular, here we proved that $\mathscr{D}=\mathscr{J}$  in $\mathscr{P\!O}\!_{\infty}(\mathbb{N}^2_{\leqslant})$.

The present paper is a continuation of \cite{Gutik-Pozdniakova-2016??}. We describe the natural partial order $\preccurlyeq$ on the semigroup $\mathscr{P\!O}\!_{\infty}(\mathbb{N}^2_{\leqslant})$ and show that it coincides with the natural partial order which is induced from symmetric inverse monoid $\mathscr{I}_{\mathbb{N}\times\mathbb{N}}$ over the set $\mathbb{N}\times\mathbb{N}$ onto the semigroup $\mathscr{P\!O}\!_{\infty}(\mathbb{N}^2_{\leqslant})$. We proved that the semigroup $\mathscr{P\!O}\!_{\infty}(\mathbb{N}^2_{\leqslant})$ is isomorphic to the semidirect product $\mathscr{P\!O}\!_{\infty}^{\,+}(\mathbb{N}^2_{\leqslant})\rtimes \mathbb{Z}_2$ of the monoid $\mathscr{P\!O}\!_{\infty}^{\,+}(\mathbb{N}^2_{\leqslant})$ of orientation-preserving monotone injective partial selfmaps of $\mathbb{N}^{2}_{\leqslant}$ with cofinite domains and images by the cyclic group $\mathbb{Z}_2$ of the order two. Also we describe the congruence $\sigma$ on the semigroup $\mathscr{P\!O}\!_{\infty}(\mathbb{N}^2_{\leqslant})$, which is generated by the natural order $\preccurlyeq$ on the semigroup $\mathscr{P\!O}\!_{\infty}(\mathbb{N}^2_{\leqslant})$: $\alpha\sigma\beta$ if and only if $\alpha$ and $\beta$ are comparable in $\left(\mathscr{P\!O}\!_{\infty}(\mathbb{N}^2_{\leqslant}),\preccurlyeq\right)$. We prove that the quotient semigroup $\mathscr{P\!O}\!_{\infty}^{\,+}(\mathbb{N}^2_{\leqslant})/\sigma$ is isomorphic to the free commutative monoid $\mathfrak{AM}_\omega$ over an infinite countable set and show that quotient semigroup $\mathscr{P\!O}\!_{\infty}(\mathbb{N}^2_{\leqslant})/\sigma$ is isomorphic to the semidirect product of the free commutative monoid $\mathfrak{AM}_\omega$ by the group $\mathbb{Z}_2$.


\medskip


The following proposition implies that the equations of the form $a\cdot x=b$ and $x\cdot c=d$ in the semigroup $\mathscr{P\!O}\!_{\infty}(\mathbb{N}^2_{\leqslant})$ have finitely many solutions. This property holds for the bicyclic monoid, many its generalizations and other semigroups (see corresponding results in \cite{Bardyla-Gutik-2016, Eberhart-Selden-1969, Gutik-Pozdnyakova-2014, Gutik-Repovs-2011, Gutik-Repovs-2012, Gutik-Repovs-2015}).

\begin{proposition}\label{proposition-2.1}
For every $\alpha,\beta\in\mathscr{P\!O}\!_{\infty}(\mathbb{N}^2_{\leqslant})$, both sets
\begin{equation*}
\left\{\chi\in\mathscr{P\!O}\!_{\infty}(\mathbb{N}^2_{\leqslant})\mid \alpha\cdot\chi=\beta\right\}\qquad \mbox{ and } \qquad
\left\{\chi\in\mathscr{P\!O}\!_{\infty}(\mathbb{N}^2_{\leqslant})\mid \chi\cdot\alpha=\beta\right\}
\end{equation*}
are finite. Consequently, every right translation and every left translation by an element of the semigroup $\mathscr{P\!O}\!_{\infty}(\mathbb{N}^2_{\leqslant})$ is a finite-to-one map.
\end{proposition}

\begin{proof}
We consider the case of the equation $\alpha\cdot\chi=\beta$. In the case of the equation $\chi\cdot\alpha=\beta$ the proof is similar.

The definition of the semigroup $\mathscr{P\!O}\!_{\infty}(\mathbb{N}^2_{\leqslant})$ and the equality $\alpha\cdot\chi=\beta$ imply that $\operatorname{dom}\beta\subseteq\operatorname{dom}\alpha$ and $\operatorname{ran}\chi\subseteq\operatorname{ran}\alpha$. Since any element of the semigroup $\mathscr{P\!O}\!_{\infty}(\mathbb{N}^2_{\leqslant})$ has a cofinite domain and a cofinite image in $\mathbb{N}\times\mathbb{N}$, we conclude that if an element $\chi_0$ satisfies the equality $\alpha\cdot\chi=\beta$ then for every other root $\chi$ of the equation $\alpha\cdot\chi=\beta$ there exist finitely many $(i,j)\in(\mathbb{N}\times\mathbb{N})\setminus\operatorname{ran}\beta$ such that one of the following conditions holds:
\begin{itemize}
  \item[$(1)$] $(i,j)\chi\neq(i,j)\chi_0$;
  \item[$(2)$] $(i,j)\chi$ is determined and $(i,j)\chi_0$ is undetermined;
  \item[$(3)$] $(i,j)\chi_0$ is determined and $(i,j)\chi$ is undetermined.
\end{itemize}
This implies that the equation $\alpha\cdot\chi=\beta$ has finitely many solutions, which completes the proof of the proposition.
\end{proof}

Later we shall describe the natural partial order ``$\preccurlyeq$'' on the semigroup $\mathscr{P\!O}\!_{\infty}(\mathbb{N}^2_{\leqslant})$. For   $\alpha,\beta\in\mathscr{P\!O}\!_{\infty}(\mathbb{N}^2_{\leqslant})$ we put
\begin{equation*}
  \alpha\preccurlyeq\beta \qquad \hbox{if and only if} \qquad \alpha=\beta\varepsilon \quad \hbox{for some} \quad \varepsilon\in E(\mathscr{P\!O}\!_{\infty}(\mathbb{N}^2_{\leqslant})).
\end{equation*}

We need the following proposition from \cite{Mitsch-1986}.

\begin{proposition}[{\cite[p.~387, Corollary]{Mitsch-1986}}]\label{proposition-2.2}
For any semigroup $S$ and its natural partial order $\preccurlyeq$ the following conditions are equivalent:
\begin{itemize}
  \item[$(i)$] $a\preccurlyeq b$;
  \item[$(ii)$] $a=wb=bz$, $az=a$ for some $w,z\in  S^1$;
  \item[$(iii)$] $a=xb=by$, $xa=ay=a$ for some $x,y\in  S^1$.
\end{itemize}
\end{proposition}

\begin{proposition}\label{proposition-2.3}
The relation $\preccurlyeq$ is the natural partial order on the semigroup $\mathscr{P\!O}\!_{\infty}(\mathbb{N}^2_{\leqslant})$.
\end{proposition}

\begin{proof}
Suppose that $\alpha=\beta\varepsilon$ for some idempotent $\varepsilon\in E(\mathscr{P\!O}\!_{\infty}(\mathbb{N}^2_{\leqslant}))$. Then we have that
\begin{equation*}
\alpha\varepsilon=(\beta\varepsilon)\varepsilon=\beta(\varepsilon\varepsilon)=\beta\varepsilon=\alpha.
\end{equation*}
Let $\iota\colon \operatorname{dom}(\beta\varepsilon)\to \operatorname{dom}(\beta\varepsilon)$ be the identity map of the set $\operatorname{dom}(\beta\varepsilon)$. Then $\iota\in E(\mathscr{P\!O}\!_{\infty}(\mathbb{N}^2_{\leqslant}))$ and the definition of the semigroup $\mathscr{P\!O}\!_{\infty}(\mathbb{N}^2_{\leqslant})$ implies that $\operatorname{dom}(\beta\varepsilon)=\operatorname{dom}(\iota\beta)$, because $\varepsilon$ is an idempotent of $\mathscr{P\!O}\!_{\infty}(\mathbb{N}^2_{\leqslant})$. This implies that $(i,j)\iota\beta=(i,j)\beta\varepsilon$ for each $(i,j)\in \operatorname{dom}(\iota\beta)$ and hence we get that $\alpha=\beta\varepsilon=\iota\beta$. Next we apply Proposition~\ref{proposition-2.2}.
\end{proof}

\begin{remark}\label{remark-2.4}
Proposition~\ref{proposition-2.3} implies that the natural partial order on the semigroup $\mathscr{P\!O}\!_{\infty}(\mathbb{N}^2_{\leqslant})$ coincides with the natural partial order which is induced from symmetric inverse monoid $\mathscr{I}_{\mathbb{N}\times\mathbb{N}}$ over the set $\mathbb{N}\times\mathbb{N}$ onto the semigroup $\mathscr{P\!O}\!_{\infty}(\mathbb{N}^2_{\leqslant})$.
\end{remark}

We define a relation $\sigma$ on the semigroup $\mathscr{P\!O}\!_{\infty}(\mathbb{N}^2_{\leqslant})$ in the following way:
\begin{equation*}
  \alpha\sigma\beta \qquad \hbox{if and only if there exists} \qquad \varepsilon\in E(\mathscr{P\!O}\!_{\infty}(\mathbb{N}^2_{\leqslant})) \quad \hbox{such that} \quad \alpha\varepsilon=\beta\varepsilon,
\end{equation*}
for $\alpha,\beta\in \mathscr{P\!O}\!_{\infty}(\mathbb{N}^2_{\leqslant})$.

\begin{proposition}\label{proposition-2.5}
For $\alpha,\beta\in \mathscr{P\!O}\!_{\infty}(\mathbb{N}^2_{\leqslant})$ the following conditions are equivalent:
\begin{itemize}
  \item[$(i)$] $\alpha\sigma\beta$;
  \item[$(ii)$] there exist $\varsigma,\upsilon \in E(\mathscr{P\!O}\!_{\infty}(\mathbb{N}^2_{\leqslant}))$ such that $\alpha\varsigma=\beta\upsilon$;
  \item[$(iii)$] there exist $\varsigma,\upsilon \in E(\mathscr{P\!O}\!_{\infty}(\mathbb{N}^2_{\leqslant}))$ such that $\alpha\varsigma=\upsilon\beta$;
  \item[$(iv)$] there exists $\iota\in E(\mathscr{P\!O}\!_{\infty}(\mathbb{N}^2_{\leqslant}))$ such that $\iota\alpha=\iota\beta$;
  \item[$(v)$] there exist $\varsigma,\upsilon \in E(\mathscr{P\!O}\!_{\infty}(\mathbb{N}^2_{\leqslant}))$ such that $\varsigma\alpha=\upsilon\beta$.
\end{itemize}
Thus $\sigma$ is a congruence on the semigroup $\mathscr{P\!O}\!_{\infty}(\mathbb{N}^2_{\leqslant})$.
\end{proposition}

\begin{proof}
Implication $(i)\Rightarrow(ii)$ is trivial.

$(ii)\Rightarrow(i)$ If we have that $\alpha\varsigma=\beta\upsilon$ for some $\varsigma,\upsilon \in E(\mathscr{P\!O}\!_{\infty}(\mathbb{N}^2_{\leqslant}))$ then $\alpha\varsigma(\varsigma\upsilon)=\beta\upsilon(\varsigma\upsilon)$. Since $\mathscr{P\!O}\!_{\infty}(\mathbb{N}^2_{\leqslant})$ is a subsemigroup of the symmetric inverse monoid $\mathscr{I}_{|\mathbb{N}\times\mathbb{N}|}$, the idempotents in the semigroup $\mathscr{P\!O}\!_{\infty}(\mathbb{N}^2_{\leqslant})$ commute and hence $\alpha(\varsigma\upsilon)=\beta(\varsigma\upsilon)$. This implies that $\alpha\sigma\beta$.

$(ii)\Rightarrow(iii)$ Suppose that $\alpha\varsigma=\beta\upsilon$ for some $\varsigma,\upsilon \in E(\mathscr{P\!O}\!_{\infty}(\mathbb{N}^2_{\leqslant}))$. Let $\iota\colon \operatorname{dom}(\beta\upsilon)\to \operatorname{dom}(\beta\upsilon)$ be the identity map of the set $\operatorname{dom}(\beta\upsilon)$. Then $\iota\in E(\mathscr{P\!O}\!_{\infty}(\mathbb{N}^2_{\leqslant}))$ and the definition of the semigroup $\mathscr{P\!O}\!_{\infty}(\mathbb{N}^2_{\leqslant})$ implies that $\operatorname{dom}(\beta\upsilon)=\operatorname{dom}(\iota\beta)$, because $\upsilon$ is an idempotent of $\mathscr{P\!O}\!_{\infty}(\mathbb{N}^2_{\leqslant})$. This implies that $(i,j)\iota\beta=(i,j)\beta\upsilon$ for each $(i,j)\in \operatorname{dom}(\iota\beta)$ and hence we get that $\alpha\varsigma=\beta\upsilon=\iota\beta$.

$(iii)\Rightarrow(ii)$ Suppose that $\alpha\varsigma=\upsilon\beta$ for some $\varsigma,\upsilon \in E(\mathscr{P\!O}\!_{\infty}(\mathbb{N}^2_{\leqslant}))$. Let $\iota\colon \operatorname{ran}(\upsilon\beta)\to \operatorname{ran}(\upsilon\beta)$ be the identity map of the set $\operatorname{ran}(\upsilon\beta)$. Then $\iota\in E(\mathscr{P\!O}\!_{\infty}(\mathbb{N}^2_{\leqslant}))$ and the definition of the semigroup $\mathscr{P\!O}\!_{\infty}(\mathbb{N}^2_{\leqslant})$ implies that $\operatorname{ran}(\upsilon\beta)=\operatorname{ran}(\beta\iota)$, because $\upsilon$ is an idempotent of $\mathscr{P\!O}\!_{\infty}(\mathbb{N}^2_{\leqslant})$. Since all elements of the semigroup $\mathscr{P\!O}\!_{\infty}(\mathbb{N}^2_{\leqslant})$ are partial bijections of $\mathbb{N}\times\mathbb{N}$ we get that $\operatorname{dom}(\upsilon\beta)=\operatorname{dom}(\beta\iota)$. This implies that $(i,j)\beta\iota=(i,j)\upsilon\beta$ for each $(i,j)\in \operatorname{dom}(\beta\iota)$ and hence we get that $\alpha\varsigma=\upsilon\beta=\beta\iota$.

The proofs of equivalences $(iii)\Leftrightarrow(iv)$ and $(iv)\Leftrightarrow(v)$ are similar.

It is obvious that $\sigma$ is a reflexive and symmetric relation on $\mathscr{P\!O}\!_{\infty}(\mathbb{N}^2_{\leqslant})$. Suppose that $\alpha\sigma\beta$ and $\beta\sigma\gamma$ in $\mathscr{P\!O}\!_{\infty}(\mathbb{N}^2_{\leqslant})$. Then there exist $\varsigma,\upsilon \in E(\mathscr{P\!O}\!_{\infty}(\mathbb{N}^2_{\leqslant}))$ such that $\alpha\varsigma=\beta\varsigma$ and $\beta\upsilon=\gamma\upsilon$. This implies that $\alpha\varsigma\upsilon=\beta\varsigma\upsilon$ and $\beta\upsilon\varsigma=\gamma\upsilon\varsigma$, and since the idempotents in $\mathscr{P\!O}\!_{\infty}(\mathbb{N}^2_{\leqslant})$ commute we get that $\alpha\varsigma\upsilon=\beta\varsigma\upsilon=\beta\upsilon\varsigma=\gamma\upsilon\varsigma$, and hence $\alpha\sigma\gamma$.

Suppose that $\alpha\sigma\beta$ for some $\alpha,\beta\in\mathscr{P\!O}\!_{\infty}(\mathbb{N}^2_{\leqslant})$. Then by $(iv)$ there exists $\iota\in E(\mathscr{P\!O}\!_{\infty}(\mathbb{N}^2_{\leqslant}))$ such that $\iota\alpha=\iota\beta$. This implies that $\iota\alpha\gamma=\iota\beta\gamma$ for each $\gamma\in\mathscr{P\!O}\!_{\infty}(\mathbb{N}^2_{\leqslant})$ and hence by item $(iv)$ we get that $(\alpha\gamma)\sigma(\beta\gamma)$. The proof of the statement that $(\gamma\alpha)\sigma(\gamma\beta)$ for each $\gamma\in\mathscr{P\!O}\!_{\infty}(\mathbb{N}^2_{\leqslant})$ is similar, and hence $\sigma$ is a congruence on the semigroup $\mathscr{P\!O}\!_{\infty}(\mathbb{N}^2_{\leqslant})$.
\end{proof}


\begin{corollary}\label{corollary-2.6}
For $\alpha,\beta\in \mathscr{P\!O}\!_{\infty}(\mathbb{N}^2_{\leqslant})$ the following condition are equivalent:
\begin{itemize}
  \item[$(i)$] $\alpha\sigma\beta$;
  \item[$(ii)$] $\alpha\varpi\sigma\beta\varpi$;
  \item[$(iii)$] $\varpi\alpha\sigma\varpi\beta$.
\end{itemize}
\end{corollary}

\begin{proof}
$(i)\Leftrightarrow (ii)$ If $\alpha\sigma\beta$ in $\mathscr{P\!O}\!_{\infty}(\mathbb{N}^2_{\leqslant})$ then by Proposition~\ref{proposition-2.5} there exists $\iota\in E(\mathscr{P\!O}\!_{\infty}(\mathbb{N}^2_{\leqslant}))$ such that $\iota\alpha=\iota\beta$. This implies that $\iota\alpha\varpi=\iota\beta\varpi$ and hence $(\alpha\varpi)\sigma(\beta\varpi)$. Conversely, if $(\alpha\varpi)\sigma(\beta\varpi)$ then by Proposition~\ref{proposition-2.5} we have that $\nu\alpha\varpi=\nu\beta\varpi$ for some $\nu\in E(\mathscr{P\!O}\!_{\infty}(\mathbb{N}^2_{\leqslant}))$, and hence
$
  \nu\alpha=\nu\alpha\varpi\varpi=\nu\beta\varpi\varpi=\nu\beta,
$
which implies that $\alpha\sigma\beta$.

The proof of $(i)\Leftrightarrow (ii)$ is similar.
\end{proof}

Also the definition of the congruence $\sigma$ on the semigroup $\mathscr{P\!O}\!_{\infty}(\mathbb{N}^2_{\leqslant})$ implies the following simple property of $\sigma$-equivalent elements of $\mathscr{P\!O}\!_{\infty}(\mathbb{N}^2_{\leqslant})$:

\begin{corollary}\label{corollary-2.6-1}
Let $\alpha,\beta$ be elements of the semigroup $\mathscr{P\!O}\!_{\infty}(\mathbb{N}^2_{\leqslant})$ such that $\alpha\sigma\beta$. Then the following assertions hold:
\begin{itemize}
  \item[$(i)$] \emph{$(\textsf{H}^{1}_{\operatorname{dom}\alpha})\alpha{\subseteq} \textsf{H}^1$} if and only if \emph{$(\textsf{H}^{1}_{\operatorname{dom}\beta})\beta{\subseteq} \textsf{H}^1$};
  \item[$(ii)$] \emph{$(\textsf{H}^{1}_{\operatorname{dom}\alpha})\alpha{\subseteq} \textsf{V}^1$} if and only if \emph{$(\textsf{H}^{1}_{\operatorname{dom}\beta})\beta{\subseteq} \textsf{V}^1$}.
\end{itemize}
\end{corollary}

We define
\begin{equation*}
  \mathscr{P\!O}\!_{\infty}^{\,+}(\mathbb{N}^2_{\leqslant})=\left\{\alpha\in\mathscr{P\!O}\!_{\infty}(\mathbb{N}^2_{\leqslant})\colon (\textsf{H}^1_{\operatorname{dom}\alpha})\alpha\subseteq\textsf{H}^1\right\}.
\end{equation*}

Then Lemma~3 and Theorem~1 from \cite{Gutik-Pozdniakova-2016??} imply that $\mathscr{P\!O}\!_{\infty}^{\,+}(\mathbb{N}^2_{\leqslant})$ is a subsemigroup of $\mathscr{P\!O}\!_{\infty}(\mathbb{N}^2_{\leqslant})$. The subsemigroup $\mathscr{P\!O}\!_{\infty}^{\,+}(\mathbb{N}^2_{\leqslant})$ is called the \emph{monoid of orientation-preserving monotone injective partial selfmaps of} $\mathbb{N}^{2}_{\leqslant}$ with cofinite domains and images.  Moreover it is obvious that $E(\mathscr{P\!O}\!_{\infty}^{\,+}(\mathbb{N}^2_{\leqslant})= E(\mathscr{P\!O}\!_{\infty}(\mathbb{N}^2_{\leqslant}))$. Also, later by $\preccurlyeq$ and $\sigma$ we denote the corresponding induced relations of the relations $\preccurlyeq$ and $\sigma$ from the semigroup $\mathscr{P\!O}\!_{\infty}(\mathbb{N}^2_{\leqslant})$ onto its subsemigroup $\mathscr{P\!O}\!_{\infty}^{\,+}(\mathbb{N}^2_{\leqslant})$.

The proofs of the following propositions are similar to those of Propositions~\ref{proposition-2.3} and \ref{proposition-2.5}, respectively.

\begin{proposition}\label{proposition-2.7}
The relation $\preccurlyeq$ is the natural partial order on the semigroup $\mathscr{P\!O}\!_{\infty}^{\,+}(\mathbb{N}^2_{\leqslant})$.
\end{proposition}

\begin{proposition}\label{proposition-2.8}
The relation $\sigma$ is a congruence on the semigroup $\mathscr{P\!O}\!_{\infty}^{\,+}(\mathbb{N}^2_{\leqslant})$.
\end{proposition}

By $\varpi$ we denote the bijective transformation of $\mathbb{N}\times\mathbb{N}$ defined by the formula $(i,j)\varpi=(j,i)$, for any $(i,j)\in\mathbb{N}\times\mathbb{N}$. It is obvious that $\varpi$ is an element of the semigroup $\mathscr{P\!O}\!_{\infty}(\mathbb{N}^2_{\leqslant})$ and $\varpi\varpi=\mathbb{I}$.

\begin{remark}\label{remark-2.9}
We observe that
\begin{itemize}
  \item[$(i)$] $\alpha\in\mathscr{P\!O}\!_{\infty}^{\,+}(\mathbb{N}^2_{\leqslant})$ if and only if $\alpha\varpi,\varpi\alpha\in \mathscr{P\!O}\!_{\infty}(\mathbb{N}^2_{\leqslant})\setminus \mathscr{P\!O}\!_{\infty}^{\,+}(\mathbb{N}^2_{\leqslant})$;
  \item[$(ii)$] $\alpha\in\mathscr{P\!O}\!_{\infty}^{\,+}(\mathbb{N}^2_{\leqslant})$ if and only if $\varpi\alpha\varpi\in \mathscr{P\!O}\!_{\infty}^{\,+}(\mathbb{N}^2_{\leqslant})$.
\end{itemize}
\end{remark}

We define a map $\mathfrak{h}\colon \mathscr{P\!O}\!_{\infty}(\mathbb{N}^2_{\leqslant})\to \mathscr{P\!O}\!_{\infty}(\mathbb{N}^2_{\leqslant})$ by the formula $(\alpha)\mathfrak{h}=\varpi\alpha\varpi$, for $\alpha\in\mathscr{P\!O}\!_{\infty}(\mathbb{N}^2_{\leqslant})$.

\begin{proposition}\label{proposition-2.10}
The map $\mathfrak{h}\colon \mathscr{P\!O}\!_{\infty}(\mathbb{N}^2_{\leqslant})\to \mathscr{P\!O}\!_{\infty}(\mathbb{N}^2_{\leqslant})$ is an automorphism of the semigroup $\mathscr{P\!O}\!_{\infty}(\mathbb{N}^2_{\leqslant})$. Moreover its restriction $\mathfrak{h}|_{\mathscr{P\!O}\!_{\infty}^{\,+}(\mathbb{N}^2_{\leqslant})}\colon \mathscr{P\!O}\!_{\infty}^{\,+}(\mathbb{N}^2_{\leqslant})\to \mathscr{P\!O}\!_{\infty}^{\,+}(\mathbb{N}^2_{\leqslant})$ is an automorphism of the subsemigroup $\mathscr{P\!O}\!_{\infty}^{\,+}(\mathbb{N}^2_{\leqslant})$.
\end{proposition}

\begin{proof}
First we show that $\mathfrak{h}\colon \mathscr{P\!O}\!_{\infty}(\mathbb{N}^2_{\leqslant})\to \mathscr{P\!O}\!_{\infty}(\mathbb{N}^2_{\leqslant})$ is a homomorphism. Fix arbitrary $\alpha,\beta\in\mathscr{P\!O}\!_{\infty}(\mathbb{N}^2_{\leqslant})$. Then we have that
\begin{equation*}
  (\alpha\beta)\mathfrak{h}=\varpi(\alpha\beta)\varpi=\varpi(\alpha\mathbb{I}\beta)\varpi=\varpi(\alpha\varpi\varpi\beta)\varpi =(\varpi\alpha\varpi)(\varpi\beta\varpi)=(\alpha)\mathfrak{h}(\beta)\mathfrak{h},
\end{equation*}
and hence $\mathfrak{h}\colon \mathscr{P\!O}\!_{\infty}(\mathbb{N}^2_{\leqslant})\to \mathscr{P\!O}\!_{\infty}(\mathbb{N}^2_{\leqslant})$ is a homomorphism.

Fix an arbitrary $\alpha\in\mathscr{P\!O}\!_{\infty}(\mathbb{N}^2_{\leqslant})$. Then the definition of $\mathfrak{h}$ implies that
\begin{equation*}
(\varpi\alpha\varpi)\mathfrak{h}=\varpi\varpi\alpha\varpi\varpi=\mathbb{I}\alpha\mathbb{I}=\alpha,
\end{equation*}
and hence the map $\mathfrak{h}\colon \mathscr{P\!O}\!_{\infty}(\mathbb{N}^2_{\leqslant})\to \mathscr{P\!O}\!_{\infty}(\mathbb{N}^2_{\leqslant})$ is surjective. Suppose that $(\alpha)\mathfrak{h}=(\beta)\mathfrak{h}$ for some $\alpha,\beta\in\mathscr{P\!O}\!_{\infty}(\mathbb{N}^2_{\leqslant})$. Then
\begin{equation*}
  \alpha=\mathbb{I}\alpha\mathbb{I}=\varpi\varpi\alpha\varpi\varpi=((\alpha)\mathfrak{h})\mathfrak{h}=((\beta)\mathfrak{h})\mathfrak{h}= \varpi\varpi\beta\varpi\varpi=\mathbb{I}\beta\mathbb{I}=\beta,
\end{equation*}
and hence the map $\mathfrak{h}\colon \mathscr{P\!O}\!_{\infty}(\mathbb{N}^2_{\leqslant})\to \mathscr{P\!O}\!_{\infty}(\mathbb{N}^2_{\leqslant})$ is injective. Thus the map $\mathfrak{h}$ is an automorphism of the semigroup $\mathscr{P\!O}\!_{\infty}(\mathbb{N}^2_{\leqslant})$.

Now, Remark~\ref{remark-2.9} implies that the restriction $\mathfrak{h}|_{\mathscr{P\!O}\!_{\infty}^{\,+}(\mathbb{N}^2_{\leqslant})}\colon \mathscr{P\!O}\!_{\infty}^{\,+}(\mathbb{N}^2_{\leqslant})\to \mathscr{P\!O}\!_{\infty}^{\,+}(\mathbb{N}^2_{\leqslant})$ is an automorphism of the semigroup $\mathscr{P\!O}\!_{\infty}^{\,+}(\mathbb{N}^2_{\leqslant})$, too.
\end{proof}

For the automorphism $\mathfrak{h}\colon \mathscr{P\!O}\!_{\infty}^{\,+}(\mathbb{N}^2_{\leqslant})\to \mathscr{P\!O}\!_{\infty}^{\,+}(\mathbb{N}^2_{\leqslant})$ of the semigroup $\mathscr{P\!O}\!_{\infty}^{\,+}(\mathbb{N}^2_{\leqslant})$ we have that $\mathfrak{h}^2=\operatorname{\textsf{Id}}_{\mathscr{P\!O}\!_{\infty}^{\,+}(\mathbb{N}^2_{\leqslant})}$ is the identity automorphism of $\mathscr{P\!O}\!_{\infty}^{\,+}(\mathbb{N}^2_{\leqslant})$. This implies that the element $\mathfrak{h}$ generates the group which is isomorphic to the cyclic group of order two $\mathbb{Z}_2$. By Proposition~4 from \cite{Gutik-Pozdniakova-2016??} the group of units $H(\mathbb{I})$ of the semigroup $\mathscr{P\!O}\!_{\infty}(\mathbb{N}^2_{\leqslant})$ is isomorphic to $\mathbb{Z}_2$. We define a map $\mathfrak{Q}$ from $H(\mathbb{I})$ into the group $\operatorname{Aut}(\mathscr{P\!O}\!_{\infty}^{\,+}(\mathbb{N}^2_{\leqslant}))$ of automorphisms of the semigroup $\mathscr{P\!O}\!_{\infty}^{\,+}(\mathbb{N}^2_{\leqslant})$ in the following way $(\mathbb{I})\mathfrak{Q}=\operatorname{\textsf{Id}}_{\mathscr{P\!O}\!_{\infty}^{\,+}(\mathbb{N}^2_{\leqslant})}$ and $(\varpi)\mathfrak{Q}=\mathfrak{h}$. It is obvious that so defined map $\mathfrak{Q}\colon H(\mathbb{I})\to\operatorname{Aut}(\mathscr{P\!O}\!_{\infty}^{\,+}(\mathbb{N}^2_{\leqslant}))$ is an injective homomorphism.

Let $S$ and $T$ be semigroups and let $\mathfrak{H}$ be a homomorphism from $T$ into the semigroup of endomorphisms $\operatorname{\textsf{End}}(S)$ of $S$, $\mathfrak{H}\colon t\mapsto \mathfrak{h}_t$. Then the Cartesian product $S\times T$ with the following semigroup operation
\begin{equation*}
  \left(s_1,t_1\right)\cdot \left(s_2,t_2\right)=\left(s_1\cdot(s_2)\mathfrak{h}_{t_1},t_1\cdot t_2\right), \qquad s_1,s_2\in S, \; t_1,t_2\in T,
\end{equation*}
is called a \emph{semidirect product} of the semigroup $S$ by $T$ and is denoted by $S\rtimes_\mathfrak{H}T$. We remark that if $1_T$ is the unit of the semigroup $T$ then $(1_T)\mathfrak{H}=\mathfrak{h}_{1_T}$ is the identity homomorphism of $S$ and in the case when $T$ is a group then $(t)\mathfrak{H}=\mathfrak{h}_{t}$ is an automorphism of $S$ for any $t\in T$.

\begin{theorem}\label{theorem-2.11}
The semigroup $\mathscr{P\!O}\!_{\infty}(\mathbb{N}^2_{\leqslant})$ is isomorphic to the semidirect product $\mathscr{P\!O}\!_{\infty}^{\,+}(\mathbb{N}^2_{\leqslant})\rtimes_{\mathfrak{Q}} H(\mathbb{I})$ of the semigroup $\mathscr{P\!O}\!_{\infty}^{\,+}(\mathbb{N}^2_{\leqslant})$ by the group $H(\mathbb{I})$.
\end{theorem}

\begin{proof}
We define a map $\mathfrak{I}\colon\mathscr{P\!O}\!_{\infty}^{\,+}(\mathbb{N}^2_{\leqslant})\rtimes_{\mathfrak{Q}} H(\mathbb{I})\to \mathscr{P\!O}\!_{\infty}(\mathbb{N}^2_{\leqslant})$ by the formula $(\alpha,g)\mathfrak{I}=\alpha g$. Then for all $\alpha_1,\alpha_2\in \mathscr{P\!O}\!_{\infty}^{\,+}(\mathbb{N}^2_{\leqslant})$ and $g_1,g_2\in H(\mathbb{I})$ we have that
\begin{equation*}
\begin{split}
  \left((\alpha_1,g_1)\cdot(\alpha_2,g_2)\right)\mathfrak{I}&= \left(\alpha_1\cdot(\alpha_2)(g_1)\mathfrak{Q},g_1\cdot g_2\right)\mathfrak{I}= \left(\alpha_1\cdot g_1\cdot\alpha_2\cdot g_1,g_1\cdot g_2\right)\mathfrak{I}=\\
    &= \alpha_1\cdot g_1\cdot\alpha_2\cdot g_1 \cdot g_1 \cdot g_2=  \alpha_1\cdot g_1\cdot \alpha_2\cdot g_2=\\
    &= (\alpha_1,g_1)\mathfrak{I}\cdot (\alpha_2,g_2)\mathfrak{I},
\end{split}
\end{equation*}
because $g^2=\mathbb{I}$ for any $g\in H(\mathbb{I})$, and hence the map $\mathfrak{I}\colon\mathscr{P\!O}\!_{\infty}^{\,+}(\mathbb{N}^2_{\leqslant})\rtimes_{\mathfrak{Q}} H(\mathbb{I})\to \mathscr{P\!O}\!_{\infty}(\mathbb{N}^2_{\leqslant})$ is a homomorphism.

By Lemma~3 from \cite{Gutik-Pozdniakova-2016??} for every $\alpha\in \mathscr{P\!O}\!_{\infty}(\mathbb{N}^2_{\leqslant})$ there exist $\alpha^+\in \mathscr{P\!O}\!_{\infty}^{\,+}(\mathbb{N}^2_{\leqslant})$ and $g_\alpha\in H(\mathbb{I})$ such that $\alpha=\alpha^+g_\alpha$. Indeed,
\begin{itemize}
  \item[$(a)$] in the case when $(\textsf{H}^1_{\operatorname{dom}\alpha})\alpha\subseteq\textsf{H}^1$ we put $\alpha^+=\alpha$ and $g_\alpha=\mathbb{I}$;
  \item[$(b)$] in the case when $(\textsf{H}^1_{\operatorname{dom}\alpha})\alpha\subseteq\textsf{V}^1$ we put $\alpha^+=\alpha\omega$ and $g_\alpha=\omega$.
\end{itemize}

Let $\alpha^+,\beta^+\in \mathscr{P\!O}\!_{\infty}^{\,+}(\mathbb{N}^2_{\leqslant})$ and $g_\alpha,g_\beta\in H(\mathbb{I})$ be such that $\alpha^+g_\alpha=(\alpha^+,g_\alpha)\mathfrak{I}=(\beta^+,g_\beta)\mathfrak{I}=\beta^+g_\beta$. Since $(\textsf{H}^1_{\operatorname{dom}\alpha^+})\alpha^+\subseteq\textsf{H}^1$ and $(\textsf{H}^1_{\operatorname{dom}\beta^+})\beta^+\subseteq\textsf{H}^1$, Lemma~3 from \cite{Gutik-Pozdniakova-2016??} implies that $g_\alpha=g_\beta$. By Proposition~4 from \cite{Gutik-Pozdniakova-2016??} the group of units $H(\mathbb{I})$ of the semigroup $\mathscr{P\!O}\!_{\infty}(\mathbb{N}^2_{\leqslant})$ is isomorphic to $\mathbb{Z}_2$ and hence   $\alpha^+=\alpha^+g_\alpha^2=\alpha^+g_\alpha g_\beta=\beta^+g_\beta^2=\beta^+$. Therefore, we get that so defined map   $\mathfrak{I}\colon\mathscr{P\!O}\!_{\infty}^{\,+}(\mathbb{N}^2_{\leqslant})\rtimes_{\mathfrak{Q}} H(\mathbb{I})\to \mathscr{P\!O}\!_{\infty}(\mathbb{N}^2_{\leqslant})$ is an isomorphism.
\end{proof}

By Theorem~2($ii_1$) from \cite{Gutik-Pozdniakova-2016??} for every $\alpha\in \mathscr{P\!O}\!_{\infty}^{\,+}(\mathbb{N}^2_{\leqslant})$ there exists a smallest positive integer $n_{\alpha}$ such that $(i,j)\alpha=(i,j)$ for each $(i,j)\in\operatorname{dom}\alpha\cap{\uparrow}(n_{\alpha},n_{\alpha})$.

\begin{lemma}\label{lemma-2.12}
For every $\alpha\in \mathscr{P\!O}\!_{\infty}^{\,+}(\mathbb{N}^2_{\leqslant})$ there exists \emph{$\alpha_{\texttt{f}}\in \mathscr{P\!O}\!_{\infty}^{\,+}(\mathbb{N}^2_{\leqslant})$} such that the following assertions hold:
\begin{itemize}
   \item[$(i)$] \emph{$\alpha\sigma\alpha_{\texttt{f}}$};
   \item[$(ii)$] \emph{$(i+1)_{\alpha_{\texttt{f}}[*,j]}-(i+1)=i_{\alpha_{\texttt{f}}[*,j]}-i$} for arbitrary \emph{$(i,j)\in \operatorname{dom}\alpha_{\texttt{f}}$} with $j<n_{\alpha}$, \emph{$(i,j)\alpha_{\texttt{f}}=\big(i_{\alpha_{\texttt{f}}[*,j]},j_{\alpha_{\texttt{f}}[i,*]}\big)$} and \emph{$(i+1,j)\alpha_{\texttt{f}}=\big((i+1)_{\alpha_{\texttt{f}}[*,j]},j_{\alpha_{\texttt{f}}[i+1,*]}\big)$}, i.e., \emph{$\alpha_{\texttt{f}}$} acts as a partial shift on the set \emph{$\textsf{H}^j$};
   \item[$(iii)$] \emph{$(j+1)_{\alpha_{\texttt{f}}[i,*]}-(j+1)=j_{\alpha_{\texttt{f}}[i,*]}-j$} for arbitrary \emph{$(i,j)\in \operatorname{dom}\alpha_{\texttt{f}}$} with $i<n_{\alpha}$, \emph{$(i,j)\alpha_{\texttt{f}}=\big(i_{\alpha_{\texttt{f}}[*,j]},j_{\alpha_{\texttt{f}}[i,*]}\big)$} and \emph{$(i,j+1)\alpha_{\texttt{f}}=\big(i_{\alpha_{\texttt{f}}[*,j+1]},(j+1)_{\alpha_{\texttt{f}}[i,*]}\big)$}, i.e., \emph{$\alpha_{\texttt{f}}$} acts as a partial shift on the set \emph{$\textsf{V}^i$}.
\end{itemize}
Moreover, there exist smallest positive integers $\widehat{h}_\alpha,\widehat{v}_\alpha\leqslant n_\alpha$ such that \emph{$(i,j)\alpha_{\texttt{f}}=(i,j)$} for arbitrary \emph{$(i,j)\in\operatorname{dom}\alpha_{\texttt{f}}$} with $i\geqslant \widehat{h}_\alpha$ and \emph{$(k,l)\alpha_{\texttt{f}}=(k,l)$} for arbitrary \emph{$(k,l)\in\operatorname{dom}\alpha_{\texttt{f}}$} with $l\geqslant \widehat{v}_\alpha$.
\end{lemma}

\begin{proof}
Fix an arbitrary element $\alpha$ of the semigroup $\mathscr{P\!O}\!_{\infty}^{\,+}(\mathbb{N}^2_{\leqslant})$. Then by Theorem~1(1) from  \cite{Gutik-Pozdniakova-2016??} we get that  $(\textsf{H}^{n}_{\operatorname{dom}\alpha})\alpha{\subseteq^*} \textsf{H}^n$ and $(\textsf{V}^{n}_{\operatorname{dom}\alpha})\alpha{\subseteq^*} \textsf{V}^n$ for any positive integer $n$. Also, the definition of the semigroup
$\mathscr{P\!O}\!_{\infty}(\mathbb{N}^2_{\leqslant})$ and  Theorem~2($ii_1$) of \cite{Gutik-Pozdniakova-2016??} imply that there exists a smallest positive integer $n_{\alpha}$ such that $(i,j)\alpha=(i,j)$ for each $(i,j)\in\operatorname{dom}\alpha\cap{\uparrow}(n_{\alpha},n_{\alpha})$, and hence for arbitrary positive integers $i,j<n_\alpha$ there exist smallest positive integers $h^i_\alpha$ and $v^j_\alpha$ such that the following conditions hold:
\begin{equation*}
\begin{split}
   & \textsf{H}^{i}_{\operatorname{ran}\alpha}\cap\left\{(p,i)\colon p\geqslant h^i_\alpha\right\}=\left\{(p,i)\colon p\geqslant h^i_\alpha\right\};\\
   & \textsf{V}^{j}_{\operatorname{ran}\alpha}\cap\left\{(j,q)\colon q\geqslant v^j_\alpha\right\}=\left\{(j,q)\colon q\geqslant v^j_\alpha\right\},
\end{split}
\end{equation*}
and
\begin{equation*}
(k,i), (j,l)\in \operatorname{dom}\alpha, \qquad  (k,i)\alpha\in \textsf{H}^i, \qquad (j,l)\alpha\in \textsf{V}^j,
\end{equation*}
for all positive integers $k\geqslant h^i_\alpha$ and $l\geqslant v^j_\alpha$.

We put
\begin{equation*}
  \bar{h}_\alpha=\max\left\{h^i_\alpha\colon i=1,\ldots,n_\alpha-1\right\} \qquad \hbox{and} \qquad \bar{v}_\alpha=\max\left\{v^j_\alpha\colon j=1,\ldots,n_\alpha-1\right\}.
\end{equation*}
The above arguments imply that
\begin{equation}\label{eq-2.1}
   \textsf{H}^{i}_{\operatorname{ran}\alpha}\cap\left\{(p,i)\colon p\geqslant\bar{h}_\alpha\right\}=\left\{(p,i)\colon p\geqslant\bar{h}_\alpha\right\};\\
\end{equation}
\begin{equation}\label{eq-2.2}
\textsf{V}^{j}_{\operatorname{ran}\alpha}\cap\left\{(j,q)\colon q\geqslant\bar{v}_\alpha\right\}=\left\{(j,q)\colon q\geqslant\bar{v}_\alpha\right\},
\end{equation}
and
\begin{equation*}
(k,i), (j,l)\in \operatorname{dom}\alpha, \qquad  (k,i)\alpha\in \textsf{H}^i, \qquad (j,l)\alpha\in \textsf{V}^j,
\end{equation*}
for all positive integers $k\geqslant\bar{h}_\alpha$ and $l\geqslant\bar{v}_\alpha$.

Next we put
\begin{equation}\label{eq-2.2a}
  D_\alpha=\left(\mathbb{N}\times\mathbb{N}\right)\setminus \left(\left\{(i,j)\colon i\leqslant \bar{h}_\alpha \hbox{~and~} j\leqslant n_\alpha \right\}\cup
 \left\{(i,j)\colon i\leqslant n_\alpha \hbox{~and~} j\leqslant\bar{v}_\alpha \right\}\right).
\end{equation}

We define $\alpha_{\!\texttt{f}}=\alpha|_{D_\alpha}$, i.e.,
\begin{equation*}
\operatorname{dom}\alpha_{\!\texttt{f}}=D_\alpha, \qquad \operatorname{ran}\alpha_{\!\texttt{f}}=(D_\alpha)\alpha \qquad \hbox{and} \qquad  (i,j)\alpha_{\!\texttt{f}}=(i,j)\alpha \quad \hbox{for all} \quad (i,j)\in \operatorname{dom}\alpha_{\!\texttt{f}}.
\end{equation*}
Since $\alpha_{\!\texttt{f}}=\varepsilon_\alpha\alpha_{\!\texttt{f}}=\varepsilon_\alpha\alpha$ for the identity partial map $\varepsilon_\alpha\colon \mathbb{N}\times\mathbb{N}\rightharpoonup\mathbb{N}\times\mathbb{N}$ with $\operatorname{dom}\varepsilon_\alpha=\operatorname{ran}\varepsilon_\alpha= D_\alpha$, Proposition~\ref{proposition-2.5} implies that $\alpha\sigma\alpha_{\!\texttt{f}}$.

Then condition (\ref{eq-2.1}) and the definition of the positive integer $\bar{h}_\alpha$ imply that
\begin{equation*}
\big(\bar{h}_\alpha+2\big)_{\alpha_{\!\texttt{f}}[*,1]}=\big(\bar{h}_\alpha+1\big)_{\alpha_{\!\texttt{f}}[*,1]}+1,
\end{equation*}
and by similar arguments and induction we have that $(i+1)_{\alpha_{\!\texttt{f}}[*,1]}=(i,1)_{\alpha_{\!\texttt{f}}[*,1]}+1$ for arbitrary $i\geqslant \bar{h}_\alpha+1$. Next, if we apply condition (\ref{eq-2.1}) and induction for arbitrary $j<n_\alpha$ then we get that $(i+1)_{\alpha_{\!\texttt{f}}[*,j]}=(i)_{\alpha_{\!\texttt{f}}[*,j]}+1$ for arbitrary $i\geqslant \bar{h}_\alpha+1$. This implies assertion $(ii)$.

The proof of item $(iii)$ is similar to $(ii)$.

The last statement of the lemma follows from the above arguments and Theorem~2($1$) from \cite{Gutik-Pozdniakova-2016??}.
\end{proof}

For every positive integer $n$ we define partial maps $\gamma_n\colon \mathbb{N}\times\mathbb{N}\rightharpoonup \mathbb{N}\times\mathbb{N}$ and $\upsilon_n\colon \mathbb{N}\times\mathbb{N}\rightharpoonup \mathbb{N}\times\mathbb{N}$ in the following way:
\begin{equation*}
\begin{array}{l}
  \operatorname{dom}\gamma_n=\mathbb{N}\times\mathbb{N}\setminus\left\{(1,i)\colon i=1,\ldots,n\right\}, \\
  \operatorname{dom}\upsilon_n=\mathbb{N}\times\mathbb{N}\setminus\left\{(i,1)\colon i=1,\ldots,n\right\},  \\
  \operatorname{ran}\gamma_n=\operatorname{ran}\upsilon_n=\mathbb{N}\times\mathbb{N}
\end{array}
\end{equation*}
and
\begin{equation*}
\begin{array}{l}
  (i,j)\gamma_n=
\left\{
  \begin{array}{ll}
    (i-1,j), & \hbox{if}~j\leqslant n; \\
    (i,j),   & \hbox{if}~j>n
  \end{array}
\right.
\qquad \hbox{for} \quad (i,j)\in\operatorname{dom}\gamma_n,\\ \\
  (i,j)\upsilon_n=
\left\{
  \begin{array}{ll}
    (i,j-1), & \hbox{if}~i\leqslant n; \\
    (i,j),   & \hbox{if}~i>n
  \end{array}
\right.
\qquad \hbox{for} \quad (i,j)\in\operatorname{dom}\upsilon_n.
\end{array}
\end{equation*}

Simple verifications show that $\gamma_n,\upsilon_n\in \mathscr{P\!O}\!_{\infty}^{\,+}(\mathbb{N}^2_{\leqslant})$ for every positive integer $n$, and moreover the subsemigroups $\left\langle\gamma_k\mid k\in\mathbb{N}\right\rangle$ and $\left\langle\upsilon_k\mid k\in\mathbb{N}\right\rangle$ of the semigroup $\mathscr{P\!O}\!_{\infty}^{\,+}(\mathbb{N}^2_{\leqslant})$, generated by the sets $\left\{\gamma_k\colon k\in\mathbb{N}\right\}$ and $\left\{\upsilon_k\colon k\in\mathbb{N}\right\}$, respectively, are isomorphic to the free Abelian semigroup over an infinite countable set.

\begin{lemma}\label{lemma-2.13}
For every $\alpha\in \mathscr{P\!O}\!_{\infty}^{\,+}(\mathbb{N}^2_{\leqslant})$ there exist finitely many elements $\gamma_{k_1},\ldots,\gamma_{k_i}$ and $\upsilon_{l_1},\ldots,\upsilon_{l_j}$ of the semigroup $\mathscr{P\!O}\!_{\infty}^{\,+}(\mathbb{N}^2_{\leqslant})$, with $k_1<\ldots<k_i$, $l_1<\ldots<l_j$, such that
\begin{equation}\label{eq-2.3}
\alpha\sigma\big(\gamma_{k_1}^{p_1}\ldots\gamma_{k_i}^{p_i}\upsilon_{l_1}^{q_1}\ldots\upsilon_{l_j}^{q_j}\big),
\end{equation}
for some positive integers $p_1,\ldots,p_i,q_1,\ldots,q_j$. Moreover if
\begin{equation*}
\alpha\sigma\big(\gamma_{k_1}^{p_1}\ldots\gamma_{k_i}^{p_i}\upsilon_{l_1}^{q_1}\ldots\upsilon_{l_j}^{q_j}\big) \qquad \hbox{and} \qquad
\beta\sigma\big(\gamma_{a_1}^{b_1}\ldots\gamma_{a_i}^{b_i}\upsilon_{c_1}^{d_1}\ldots\upsilon_{c_j}^{d_j}\big)
\end{equation*}
for some $\alpha,\beta\in \mathscr{P\!O}\!_{\infty}^{\,+}(\mathbb{N}^2_{\leqslant})$ then $(\alpha,\beta)\notin\sigma$ if and only if
\begin{equation*}
\iota\gamma_{k_1}^{p_1}\ldots\gamma_{k_i}^{p_i}\upsilon_{l_1}^{q_1}\ldots\upsilon_{l_j}^{q_j}\neq \iota\gamma_{a_1}^{b_1}\ldots\gamma_{a_i}^{b_i}\upsilon_{c_1}^{d_1}\ldots\upsilon_{c_j}^{d_j}
\end{equation*}
for any idempotent $\iota\in\mathscr{P\!O}\!_{\infty}^{\,+}(\mathbb{N}^2_{\leqslant})$.
\end{lemma}

\begin{proof}
Fix an arbitrary element $\alpha$ of the semigroup $\mathscr{P\!O}\!_{\infty}^{\,+}(\mathbb{N}^2_{\leqslant})$. Let $\alpha_{\!\texttt{f}}$ be the element of $\mathscr{P\!O}\!_{\infty}^{\,+}(\mathbb{N}^2_{\leqslant})$ defined in the proof of Lemma~\ref{lemma-2.12}. By Theorem~3 from \cite{Gutik-Pozdniakova-2016??} and the second statement of Lemma~\ref{lemma-2.12} there exist smallest positive integers $\widehat{h}_\alpha,\widehat{v}_\alpha\leqslant n_\alpha$ such that $(i,j)\alpha_{\!\texttt{f}}=(i,j)$ for arbitrary $(i,j)\in\operatorname{dom}\alpha_{\!\texttt{f}}$ with $i\geqslant \widehat{h}_\alpha$ and $(k,l)\alpha_{\!\texttt{f}}=(k,l)$ for arbitrary $(k,l)\in\operatorname{dom}\alpha_{\!\texttt{f}}$ with $l\geqslant \widehat{v}_\alpha$.

By Lemma~\ref{lemma-2.12} and Theorem~1(1) of \cite{Gutik-Pozdniakova-2016??} we have that
\begin{equation*}
\big(j,\widehat{h}_\alpha-1\big)\alpha_{\!\texttt{f}}=\big(j_{\alpha_{\!\texttt{f}}[*,\widehat{h}_\alpha-1]},\widehat{h}_\alpha-1\big)<(j,\widehat{h}_\alpha-1) \quad \hbox{and} \quad (j+1)_{\alpha_{\!\texttt{f}}[*,\widehat{h}_\alpha-1]}-j_{\alpha_{\!\texttt{f}}[*,\widehat{h}_\alpha-1]}=1,
\end{equation*}
for arbitrary $\big(j,\widehat{h}_\alpha-1\big), \big(j+1,\widehat{h}_\alpha-1\big)\in\operatorname{dom}\alpha_{\!\texttt{f}}$. Then we put $p_{\widehat{h}_\alpha-1}=j-j_{\alpha_{\!\texttt{f}}[*,\widehat{h}_\alpha-1]}$. Next, for $s=2,\ldots,\widehat{h}_\alpha-2$ we define integers $p_{\widehat{h}_\alpha-s},\ldots,p_{1}$ by induction,
\begin{equation*}
p_{\widehat{h}_\alpha-s}= j-j_{\alpha_{\!\texttt{f}}[*,\widehat{h}_\alpha-s]}-\left(p_{\widehat{h}_\alpha-1}+\ldots+p_{\widehat{h}_\alpha-s+1}\right),
\end{equation*}
where $\big(j,\widehat{h}_\alpha-s\big)\alpha_{\!\texttt{f}}=
\big(j_{\alpha_{\!\texttt{f}}[*,\widehat{h}_\alpha-s]},\widehat{h}_\alpha-s\big)\leqslant(j,\widehat{h}_\alpha-s)$ for arbitrary $\big(j,\widehat{h}_\alpha-s\big)\in\operatorname{dom}\alpha_{\!\texttt{f}}$.

Similarly, by Lemma~\ref{lemma-2.12} and Theorem~1(1) of \cite{Gutik-Pozdniakova-2016??} we have that
\begin{equation*}
\big(\widehat{v}_\alpha-1, i\big)\alpha_{\!\texttt{f}}= \big(\widehat{v}_\alpha-1,i_{\alpha_{\!\texttt{f}}[\widehat{v}_\alpha-1,*]}\big)<(\widehat{v}_\alpha-1,i) \quad \hbox{and} \quad (i+1)_{\alpha_{\!\texttt{f}}[\widehat{v}_\alpha-1,*]}-i_{\alpha_{\!\texttt{f}}[\widehat{v}_\alpha-1,*]}=1,
\end{equation*}
for arbitrary $\big(\widehat{v}_\alpha-1, i\big), \big(\widehat{v}_\alpha-1, i+1\big)\in\operatorname{dom}\alpha_{\!\texttt{f}}$. Then we put $q_{\widehat{v}_\alpha-1}=i-i_{\alpha_{\!\texttt{f}}[\widehat{v}_\alpha-1,*]}$. Next, for $t=2,\ldots,\widehat{v}_\alpha-2$ we define integers $q_{\widehat{v}_\alpha-t},\ldots,q_{1}$ by induction
\begin{equation*}
q_{\widehat{v}_\alpha-t}= i-i_{\alpha_{\!\texttt{f}}[\widehat{v}_\alpha-t,*]}-\left(q_{\widehat{v}_\alpha-1}+\ldots+q_{\widehat{v}_\alpha-t+1}\right),
\end{equation*}
where $\big(\widehat{v}_\alpha-t,i\big)\alpha_{\!\texttt{f}}= \big(\widehat{v}_\alpha-t,i_{\alpha_{\!\texttt{f}}[\widehat{v}_\alpha-t,*]}\big)\leqslant\big(\widehat{v}_\alpha-t, i\big)$ for arbitrary $\big(\widehat{v}_\alpha-t,i\big)\in\operatorname{dom}\alpha_{\!\texttt{f}}$.

For any $\alpha\in\mathscr{P\!O}\!_{\infty}^{\,+}(\mathbb{N}^2_{\leqslant})$ put $\varepsilon_\alpha\colon \mathbb{N}\times\mathbb{N}$ be the identity partial map with $\operatorname{dom}\varepsilon_\alpha=\operatorname{ran}\varepsilon_\alpha= D_\alpha$, where the set $D_\alpha$ is defined by formula (\ref{eq-2.2a}).
Simple verification shows that $\varepsilon_\alpha\alpha= \varepsilon_\alpha\big(\gamma_{1}^{p_1}\ldots\gamma_{\widehat{h}_\alpha-1}^{p_{\widehat{h}_\alpha-1}} \upsilon_{1}^{q_1}\ldots\upsilon_{l_j}^{q_{\widehat{v}_\alpha-1}}\big)$ and hence
\begin{equation*}
\alpha\sigma\big(\gamma_{1}^{p_1}\ldots\gamma_{\widehat{h}_\alpha-1}^{p_{\widehat{h}_\alpha-1}} \upsilon_{1}^{q_1}\ldots\upsilon_{l_j}^{q_{\widehat{v}_\alpha-1}}\big),
\end{equation*}
which implies that relation (\ref{eq-2.3}) holds.

Since $\gamma^0_{m}=\upsilon^0_{m}=\mathbb{I}$ for any positive integer $m$, without loss of generality we may assume that $p_1,\ldots,{p_i},q_1,\ldots,q_j$ are positive integers in formula (\ref{eq-2.3}).

Also, the last statement of the lemma follows from the definition of the congruence $\sigma$ on the semigroup $\mathscr{P\!O}\!_{\infty}^{\,+}(\mathbb{N}^2_{\leqslant})$.
\end{proof}

\begin{lemma}\label{lemma-2.14}
Let be $\alpha\sigma\big(\gamma_{k_1}^{p_1}\ldots\gamma_{k_i}^{p_i}\upsilon_{l_1}^{q_1}\ldots\upsilon_{l_j}^{q_j}\big)$ for $\alpha\in \mathscr{P\!O}\!_{\infty}^{\,+}(\mathbb{N}^2_{\leqslant})$ and positive integers $p_1,\ldots,p_i$, $q_1,\ldots,q_j$, $k_1<\ldots<k_i$, $l_1<\ldots<l_j$. Then there exists an idempotent
$\widehat{\varepsilon}_\alpha\in\mathscr{P\!O}\!_{\infty}^{\,+}(\mathbb{N}^2_{\leqslant})$ such that
\begin{equation*}
  \widehat{\varepsilon}_\alpha\alpha= \widehat{\varepsilon}_\alpha\gamma_{k_1}^{p_1}\ldots\gamma_{k_i}^{p_i}\upsilon_{l_1}^{q_1}\ldots\upsilon_{l_j}^{q_j}=
  \widehat{\varepsilon}_\alpha\upsilon_{l_1}^{q_1}\ldots\upsilon_{l_j}^{q_j}\gamma_{k_1}^{p_1}\ldots\gamma_{k_i}^{p_i}.
\end{equation*}
\end{lemma}

\begin{proof}
Put
\begin{equation*}
\overline{m}_\alpha=n_\alpha+\bar{h}_\alpha+\bar{v}_\alpha+p_1+\ldots+p_i+q_1+\ldots+q_j,
\end{equation*}
where $\bar{h}_\alpha$ and $\bar{v}_\alpha$ are the positive integers defined in the proof of Lemma~\ref{lemma-2.12}. We define the identity partial map $\widehat{\varepsilon}_\alpha\colon \mathbb{N}\times\mathbb{N}\rightharpoonup\mathbb{N}\times\mathbb{N}$ with $\operatorname{dom}\widehat{\varepsilon}_\alpha=\operatorname{ran}\widehat{\varepsilon}_\alpha=M_\alpha$, where
\begin{equation*}
  M_\alpha=\left(\mathbb{N}\times\mathbb{N}\right)\setminus \left(\left\{(i,j)\colon i\leqslant \overline{m}_\alpha \hbox{~and~} j\leqslant \overline{m}_\alpha \right\}\right).
\end{equation*}
Then $\widehat{\varepsilon}_\alpha\preccurlyeq\varepsilon_\alpha$ where $\varepsilon_\alpha$ is the idempotent of the semigroup $\mathscr{P\!O}\!_{\infty}^{\,+}(\mathbb{N}^2_{\leqslant})$ defined in the proof of Lem\-ma~\ref{lemma-2.12}. This implies that
\begin{equation*}
  \widehat{\varepsilon}_\alpha\alpha=\widehat{\varepsilon}_\alpha\varepsilon_\alpha\alpha= \widehat{\varepsilon}_\alpha\varepsilon_\alpha\gamma_{k_1}^{p_1}\ldots\gamma_{k_i}^{p_i}\upsilon_{l_1}^{q_1}\ldots\upsilon_{l_j}^{q_j}= \widehat{\varepsilon}_\alpha\gamma_{k_1}^{p_1}\ldots\gamma_{k_i}^{p_i}\upsilon_{l_1}^{q_1}\ldots\upsilon_{l_j}^{q_j},
\end{equation*}
and the equlity
\begin{equation*}
  \widehat{\varepsilon}_\alpha\gamma_{k_1}^{p_1}\ldots\gamma_{k_i}^{p_i}\upsilon_{l_1}^{q_1}\ldots\upsilon_{l_j}^{q_j}=
  \widehat{\varepsilon}_\alpha\upsilon_{l_1}^{q_1}\ldots\upsilon_{l_j}^{q_j}\gamma_{k_1}^{p_1}\ldots\gamma_{k_i}^{p_i}
\end{equation*}
follows from the definition of the idempotent $\widehat{\varepsilon}_\alpha\in\mathscr{P\!O}\!_{\infty}^{\,+}(\mathbb{N}^2_{\leqslant})$.
\end{proof}

The following theorem describes the quotient semigroup $\mathscr{P\!O}\!_{\infty}^{\,+}(\mathbb{N}^2_{\leqslant})/\sigma$.

\begin{theorem}\label{theorem-2.15}
The quotient semigroup $\mathscr{P\!O}\!_{\infty}^{\,+}(\mathbb{N}^2_{\leqslant})/\sigma$ is isomorphic to the free commutative monoid $\mathfrak{AM}_\omega$ over an infinite countable set.
\end{theorem}

\begin{proof}
Let $X=\left\{a_i\colon i\in\mathbb{N}\right\}\cup\left\{b_j\colon j\in\mathbb{N}\right\}$ be a countable infinite set.

We define the map $\mathfrak{H}_{\sigma}\colon \mathscr{P\!O}\!_{\infty}^{\,+}(\mathbb{N}^2_{\leqslant})\rightarrow\mathfrak{AM}_X$ in the following way:
\begin{itemize}
  \item[$(a)$] if $\alpha\sigma\big(\gamma_{k_1}^{p_1}\ldots\gamma_{k_i}^{p_i}\upsilon_{l_1}^{q_1}\ldots\upsilon_{l_j}^{q_j}\big)$ for some  positive integers $p_1,\ldots,p_i$, $q_1,\ldots,q_j$, $k_1<\ldots<k_i$, $l_1<\ldots<l_j$, then
      \begin{equation*}
        (\alpha)\mathfrak{H}_{\sigma}=
        \big(\gamma_{k_1}^{p_1}\ldots\gamma_{k_i}^{p_i}\upsilon_{l_1}^{q_1}\ldots\upsilon_{l_j}^{q_j}\big)\mathfrak{H}_{\sigma}=
        a_{k_1}^{p_1}\ldots a_{k_i}^{p_i}b_{l_1}^{q_1}\ldots b_{l_j}^{q_j};
      \end{equation*}
  \item[$(b)$] $(\mathbb{I})\mathfrak{H}_{\sigma}=e$, where $e$ is the unit of the free commutative monoid $\mathfrak{AM}_X$.
\end{itemize}

Then Lemmas~\ref{lemma-2.13} and~\ref{lemma-2.14} imply that $(\alpha)\mathfrak{H}_{\sigma}=(\beta)\mathfrak{H}_{\sigma}$ if and only if $\alpha\sigma\beta$ in $\mathscr{P\!O}\!_{\infty}^{\,+}(\mathbb{N}^2_{\leqslant})$ and hence the quotient semigroup $\mathscr{P\!O}\!_{\infty}^{\,+}(\mathbb{N}^2_{\leqslant})/\sigma$ is isomorphic to the free commutative monoid $\mathfrak{AM}_X$.
\end{proof}

 The following corollary of Theorem~\ref{theorem-2.15} shows that the semigroup $\mathscr{P\!O}\!_{\infty}^{\,+}(\mathbb{N}^2_{\leqslant})$ has infinitely many congruences similar as the free commutative monoid $\mathfrak{AM}_\omega$ over an infinite countable set.

\begin{corollary}\label{corollary-2.16}
Every countable \emph{(}infinite or finite\emph{)} commutative monoid is a homomorphic image of the semigroup $\mathscr{P\!O}\!_{\infty}^{\,+}(\mathbb{N}^2_{\leqslant})$.
\end{corollary}

Its obvious that every non-unit element $u$ of the free commutative monoid $\mathfrak{AM}_\omega$ over the infinite countable set $\left\{a_i\colon i\in\omega\right\}\cup\left\{b_j\colon j\in\omega\right\}$ can be represented in the form $u=a_1^{i_1}\ldots a_k^{i_k}b_1^{j_1}\ldots b_l^{j_l}$, where $i_1,\ldots,i_k,j_1,\ldots,l_l$ are positive integers. We define a map $\mathfrak{f}\colon \mathfrak{AM}_\omega\to \mathfrak{AM}_\omega$  by the formula
\begin{equation}\label{eq-2.4}
(a_1^{i_1}\ldots a_k^{i_k}b_1^{j_1}\ldots b_l^{j_l})\mathfrak{f}=a_1^{j_1}\ldots a_l^{j_l}b_1^{i_1}\ldots b_k^{i_k},
\end{equation}
for $u=a_1^{i_1}\ldots a_k^{i_k}b_1^{j_1}\ldots b_l^{j_l}\in\mathfrak{AM}_\omega$ and $(e)\mathfrak{f}=e$, for unit element $e$ of $\mathfrak{AM}_\omega$.

\begin{proposition}\label{proposition-2.18}
The map $\mathfrak{f}\colon \mathfrak{AM}_\omega\to \mathfrak{AM}_\omega$ is an automorphism of the free commutative monoid $\mathfrak{AM}_\omega$.
\end{proposition}

\begin{proof}
First we show that $\mathfrak{f}\colon \mathfrak{AM}_\omega\to \mathfrak{AM}_\omega$ is a homomorphism. Fix arbitrary elements $u,v\in\mathfrak{AM}_\omega$.  Without loss of generality we may assume that
\begin{equation*}
u=a_1^{i_1}\ldots a_p^{i_p}b_1^{j_1}\ldots b_p^{j_p}\qquad \hbox{and} \qquad v=a_1^{s_1}\ldots a_p^{s_p}b_1^{t_1}\ldots b_p^{t_p}
\end{equation*}
for some non-negative integers $p,i_1,\ldots, i_p,j_1,\ldots,j_p,s_1,\ldots,s_p,t_1,\ldots,t_p$, where $a^i=b^i=e$ for $i=0$.

Then we have that
\begin{equation*}
\begin{split}
  (uv)\mathfrak{f}
    & = (a_1^{i_1}\ldots a_p^{i_p}b_1^{j_1}\ldots b_p^{j_p}a_1^{s_1}\ldots a_p^{s_p}b_1^{t_1}\ldots b_p^{t_p})\mathfrak{f}= \\
    & = (a_1^{i_1+s_1}\ldots a_p^{i_p+s_p}b_1^{j_1+t_1}\ldots b_p^{j_p+t_p})\mathfrak{f}=\\
    & = a_1^{j_1+t_1}\ldots a_p^{j_p+t_p}b_1^{i_1+s_1}\ldots b_p^{i_p+s_p}=\\
    & = a_1^{j_1}\ldots a_p^{j_p}b_1^{i_1}\ldots b_p^{i_p}a_1^{t_1}\ldots a_p^{t_p}b_1^{s_1}\ldots b_p^{s_p}=\\
    & = (a_1^{i_1}\ldots a_p^{i_p}b_p^{j_1}\ldots b_l^{j_p})\mathfrak{h} (a_1^{s_1}\ldots a_p^{s_p}b_1^{t_1}\ldots b_p^{t_p})\mathfrak{f}= \\
    & = (u)\mathfrak{f}(v)\mathfrak{f}.
\end{split}
\end{equation*}

It is obvious that $\mathfrak{f}\colon \mathfrak{AM}_\omega\to \mathfrak{AM}_\omega$ is a bijective map
and hence $\mathfrak{f}\colon \mathfrak{AM}_\omega\to \mathfrak{AM}_\omega$ is an automorph\-ism.
\end{proof}

The relationships between elements of the subsemigroup $\left\langle\gamma_k\mid k\in\mathbb{N}\right\rangle$ and of the subsemigroup $\left\langle\upsilon_k\mid k\in\mathbb{N}\right\rangle$ in $\mathscr{P\!O}\!_{\infty}^{\,+}(\mathbb{N}^2_{\leqslant})$ is described by the following proposition.

We observe that the cyclic group $\mathbb{Z}_2$ acts on the free commutative monoid $\mathfrak{AM}_\omega$ over the infinite countable set $\left\{a_i\colon i\in\omega\right\}\cup\left\{b_j\colon j\in\omega\right\}$ in the following way
\begin{equation*}
\mathfrak{AM}_\omega\times\mathbb{Z}_2\to \mathfrak{AM}_\omega\colon (u,g)\mapsto v=
\left\{
  \begin{array}{ll}
    u,               & \hbox{if~} g=\bar{0};\\
    (u)\mathfrak{f}, & \hbox{if~} g=\bar{1},
  \end{array}
\right.
\end{equation*}
where the map $\mathfrak{f}\colon \mathfrak{AM}_\omega\to \mathfrak{AM}_\omega$ is defined by formula(\ref{eq-2.4}). By Proposition~\ref{proposition-2.18} the map $\mathfrak{f}$ is an auto\-mor\-phism of the free commutative monoid $\mathfrak{AM}_\omega$.

\begin{proposition}\label{proposition-2.19}
Let $p_1,\ldots,p_i$, $k_1,\ldots,k_i$  be some  positive integers such that $k_1<\ldots<k_i$. Then the following assertions hold:
\begin{itemize}
  \item[$(i)$] $\varpi\gamma_{k_1}^{p_1}\ldots\gamma_{k_i}^{p_i}\varpi=\upsilon_{k_1}^{p_1}\ldots\upsilon_{k_i}^{p_i}$;
  \item[$(ii)$] $\gamma_{k_1}^{p_1}\ldots\gamma_{k_i}^{p_i}\varpi=\varpi\upsilon_{k_1}^{p_1}\ldots\upsilon_{k_i}^{p_i}$;
  \item[$(iii)$] $\varpi\gamma_{k_1}^{p_1}\ldots\gamma_{k_i}^{p_i}=\upsilon_{k_1}^{p_1}\ldots\upsilon_{k_i}^{p_i}\varpi$;
  \item[$(iv)$] $\varpi\upsilon_{k_1}^{p_1}\ldots\upsilon_{k_i}^{p_i}\varpi=\gamma_{k_1}^{p_1}\ldots\gamma_{k_i}^{p_i}$.
\end{itemize}
\end{proposition}

\begin{proof}
Assertion $(i)$ follows from the definitions of the elements of the semigroups $\left\langle\gamma_k\mid k\in\mathbb{N}\right\rangle$ and $\left\langle\upsilon_k\mid k\in\mathbb{N}\right\rangle$. Other assertions follow from $(i)$ and the equality $\varpi\varpi=\mathbb{I}$.
\end{proof}

Later we assume that $\mathbb{Z}_2=\left\{\bar{0},\bar{1}\right\}$.

The following theorem describes the quotient semigroup $\mathscr{P\!O}\!_{\infty}(\mathbb{N}^2_{\leqslant})/\sigma$.

\begin{theorem}\label{theorem-2.20}
The semigroup $\mathscr{P\!O}\!_{\infty}(\mathbb{N}^2_{\leqslant})/\sigma$ is isomorphic to the semidirect product $\mathfrak{AM}_\omega\rtimes_{\mathfrak{Q}} \mathbb{Z}_2$ of the free commutative monoid $\mathfrak{AM}_\omega$ over an infinite countable set by the cyclic group $\mathbb{Z}_2$.
\end{theorem}

\begin{proof}
We define a map $\mathfrak{I}\colon \mathscr{P\!O}\!_{\infty}(\mathbb{N}^2_{\leqslant})/\sigma \to \mathfrak{AM}_\omega\rtimes_{\mathfrak{Q}} \mathbb{Z}_2\colon x\mapsto (u,g)$ in the following way. Let $\mathfrak{P}_{\sigma}\colon \mathscr{P\!O}\!_{\infty}(\mathbb{N}^2_{\leqslant})\rightarrow\mathscr{P\!O}\!_{\infty}(\mathbb{N}^2_{\leqslant})/\sigma$ be the natural homomorphism generated by the congruence $\sigma$ on the semigroup $\mathscr{P\!O}\!_{\infty}(\mathbb{N}^2_{\leqslant})$. Then for every $x\in \mathscr{P\!O}\!_{\infty}(\mathbb{N}^2_{\leqslant})/\sigma$ for any $\alpha_x\in\mathscr{P\!O}\!_{\infty}(\mathbb{N}^2_{\leqslant})$ such that $(\alpha_x)\mathfrak{P}_{\sigma}=x$ only one of the following conditions holds:
\begin{itemize}
  \item[(1)]  $(\textsf{H}^{1}_{\operatorname{dom}\alpha_x})\alpha_x{\subseteq} \textsf{H}^1$;
  \item[(2)]  $(\textsf{H}^{1}_{\operatorname{dom}\alpha_x})\alpha_x{\subseteq} \textsf{V}^1$.
\end{itemize}
We put
\begin{equation}\label{eq-2.6}
  (x)\mathfrak{I}=
\left\{
  \begin{array}{cl}
    \left((\alpha_x)\mathfrak{H}_{\sigma},\bar{0}\right),       & \hbox{if~} (\textsf{H}^{1}_{\operatorname{dom}\alpha_x})\alpha_x{\subseteq} \textsf{H}^1;\\
    \left((\alpha_x\varpi)\mathfrak{H}_{\sigma},\bar{1}\right), & \hbox{if~} (\textsf{H}^{1}_{\operatorname{dom}\alpha_x})\alpha_x{\subseteq} \textsf{V}^1.
  \end{array}
\right.
\end{equation}
for all $\alpha_x\in\mathscr{P\!O}\!_{\infty}(\mathbb{N}^2_{\leqslant})$ with $(\alpha_x)\mathfrak{P}_{\sigma}=x$. Then the definition of the congruence $\sigma$ on the semigroup $\mathscr{P\!O}\!_{\infty}(\mathbb{N}^2_{\leqslant})$ and Corollary~\ref{corollary-2.6-1} imply that the map
$\mathfrak{I}\colon \mathscr{P\!O}\!_{\infty}(\mathbb{N}^2_{\leqslant})/\sigma\to\mathfrak{AM}_\omega\times\mathbb{Z}_2$ is well defined.

We observe that formula~(\ref{eq-2.6}) implies that $(x_{\mathbb{I}})\mathfrak{I}=\left(e,\bar{0}\right)$ for $x_{\mathbb{I}}=(\mathbb{I})\mathfrak{P}_{\sigma}$ and $(x_{\varpi})\mathfrak{I}=\left(e,\bar{1}\right)$ for $x_{\varpi}=(\varpi)\mathfrak{P}_{\sigma}$. Hence we have that
\begin{equation*}
\begin{split}
  (x)\mathfrak{I}\cdot(x_\mathbb{I})\mathfrak{I}
& =
\left\{
  \begin{array}{cl}
    \left((\alpha_x)\mathfrak{H}_{\sigma},\bar{0}\right)\cdot\left(e,\bar{0}\right),       & \hbox{if~} (\textsf{H}^{1}_{\operatorname{dom}\alpha_x})\alpha_x{\subseteq} \textsf{H}^1;\\
    \left((\alpha_x\varpi)\mathfrak{H}_{\sigma},\bar{1}\right)\cdot\left(e,\bar{0}\right), & \hbox{if~} (\textsf{H}^{1}_{\operatorname{dom}\alpha_x})\alpha_x{\subseteq} \textsf{V}^1
  \end{array}
\right.
= \\
    & =
\left\{
  \begin{array}{cl}
    \left((\alpha_x)\mathfrak{H}_{\sigma}\cdot e,\bar{0}\cdot\bar{0}\right),       & \hbox{if~} (\textsf{H}^{1}_{\operatorname{dom}\alpha_x})\alpha_x{\subseteq} \textsf{H}^1;\\
    \left((\alpha_x\varpi)\mathfrak{H}_{\sigma}\cdot e,\bar{1}\cdot\bar{0}\right), & \hbox{if~} (\textsf{H}^{1}_{\operatorname{dom}\alpha_x})\alpha_x{\subseteq} \textsf{V}^1
  \end{array}
\right.
= \\
    & =
\left\{
  \begin{array}{cl}
    \left((\alpha_x)\mathfrak{H}_{\sigma},\bar{0}\right),       & \hbox{if~} (\textsf{H}^{1}_{\operatorname{dom}\alpha_x})\alpha_x{\subseteq} \textsf{H}^1;\\
    \left((\alpha_x\varpi)\mathfrak{H}_{\sigma},\bar{1}\right), & \hbox{if~} (\textsf{H}^{1}_{\operatorname{dom}\alpha_x})\alpha_x{\subseteq} \textsf{V}^1
  \end{array}
\right.
= \\
    & = (x)\mathfrak{I}
\end{split}
\end{equation*}
and
\begin{equation*}
\begin{split}
  (x_\mathbb{I})\mathfrak{I}\cdot(x)\mathfrak{I}
& =
\left\{
  \begin{array}{cl}
   \left(e,\bar{0}\right)\cdot\left((\alpha_x)\mathfrak{H}_{\sigma},\bar{0}\right),       & \hbox{if~} (\textsf{H}^{1}_{\operatorname{dom}\alpha_x})\alpha_x{\subseteq} \textsf{H}^1;\\
   \left(e,\bar{0}\right)\cdot\left((\alpha_x\varpi)\mathfrak{H}_{\sigma},\bar{1}\right), & \hbox{if~} (\textsf{H}^{1}_{\operatorname{dom}\alpha_x})\alpha_x{\subseteq} \textsf{V}^1
  \end{array}
\right.
= \\
    & =
\left\{
  \begin{array}{cl}
    \left(e\cdot (\alpha_x)\mathfrak{H}_{\sigma},\bar{0}\cdot\bar{0}\right),       & \hbox{if~} (\textsf{H}^{1}_{\operatorname{dom}\alpha_x})\alpha_x{\subseteq} \textsf{H}^1;\\
    \left(e\cdot (\alpha_x\varpi)\mathfrak{H}_{\sigma},\bar{0}\cdot\bar{1}\right), & \hbox{if~} (\textsf{H}^{1}_{\operatorname{dom}\alpha_x})\alpha_x{\subseteq} \textsf{V}^1
  \end{array}
\right.
= \\
    & =
\left\{
  \begin{array}{cl}
    \left((\alpha_x)\mathfrak{H}_{\sigma},\bar{0}\right),       & \hbox{if~} (\textsf{H}^{1}_{\operatorname{dom}\alpha_x})\alpha_x{\subseteq} \textsf{H}^1;\\
    \left((\alpha_x\varpi)\mathfrak{H}_{\sigma},\bar{1}\right), & \hbox{if~} (\textsf{H}^{1}_{\operatorname{dom}\alpha_x})\alpha_x{\subseteq} \textsf{V}^1
  \end{array}
\right.
= \\
    & = (x)\mathfrak{I}.
\end{split}
\end{equation*}
Also, since $\sigma$ is congruence on $\mathscr{P\!O}\!_{\infty}(\mathbb{N}^2_{\leqslant})$, we get
\begin{equation*}
\begin{split}
  (x)\mathfrak{I}\cdot(x_\varpi)\mathfrak{I}
& =
\left\{
  \begin{array}{cl}
    \left((\alpha_x)\mathfrak{H}_{\sigma},\bar{0}\right)\cdot\left(e,\bar{1}\right),       & \hbox{if~} (\textsf{H}^{1}_{\operatorname{dom}\alpha_x})\alpha_x{\subseteq} \textsf{H}^1;\\
    \left((\alpha_x\varpi)\mathfrak{H}_{\sigma},\bar{1}\right)\cdot\left(e,\bar{1}\right), & \hbox{if~} (\textsf{H}^{1}_{\operatorname{dom}\alpha_x})\alpha_x{\subseteq} \textsf{V}^1
  \end{array}
\right.
= \\
    & =
\left\{
  \begin{array}{cl}
    \left((\alpha_x)\mathfrak{H}_{\sigma}\cdot e,\bar{0}\cdot\bar{1}\right),       & \hbox{if~} (\textsf{H}^{1}_{\operatorname{dom}\alpha_x})\alpha_x{\subseteq} \textsf{H}^1;\\
    \left((\alpha_x\varpi)\mathfrak{H}_{\sigma}\cdot e,\bar{1}\cdot\bar{1}\right), & \hbox{if~} (\textsf{H}^{1}_{\operatorname{dom}\alpha_x})\alpha_x{\subseteq} \textsf{V}^1
  \end{array}
\right.
= \\
& =
\left\{
  \begin{array}{cl}
    \left((\alpha_x)\mathfrak{H}_{\sigma},\bar{1}\right),       & \hbox{if~} (\textsf{H}^{1}_{\operatorname{dom}\alpha_x})\alpha_x{\subseteq} \textsf{H}^1;\\
    \left((\alpha_x\varpi)\mathfrak{H}_{\sigma},\bar{0}\right), & \hbox{if~} (\textsf{H}^{1}_{\operatorname{dom}\alpha_x})\alpha_x{\subseteq} \textsf{V}^1
  \end{array}
\right.
= \\
& =
\left\{
  \begin{array}{cl}
    \left((\alpha_x\varpi\varpi)\mathfrak{H}_{\sigma},\bar{1}\right),       & \hbox{if~} (\textsf{H}^{1}_{\operatorname{dom}\alpha_x})\alpha_x{\subseteq} \textsf{H}^1;\\
    \left((\alpha_x\varpi)\mathfrak{H}_{\sigma},\bar{0}\right), & \hbox{if~} (\textsf{H}^{1}_{\operatorname{dom}\alpha_x})\alpha_x{\subseteq} \textsf{V}^1
  \end{array}
\right.
= \\
     & =
\left\{
  \begin{array}{cl}
    \left(((\alpha_x\varpi)\varpi)\mathfrak{H}_{\sigma},\bar{1}\right),       & \hbox{if~} (\textsf{H}^{1}_{\operatorname{dom}(\alpha_x\varpi)})\alpha_x\varpi{\subseteq} \textsf{V}^1;\\
    \left((\alpha_x\varpi)\mathfrak{H}_{\sigma},\bar{0}\right), & \hbox{if~} (\textsf{H}^{1}_{\operatorname{dom}(\alpha_x\varpi)})\alpha_x\varpi{\subseteq} \textsf{H}^1
  \end{array}
\right.
= \\
    & = (x\cdot x_\varpi)\mathfrak{I}
\end{split}
\end{equation*}
and
\begin{itemize}
  \item[$(i)$] in the case when $(\textsf{H}^{1}_{\operatorname{dom}\alpha_x})\alpha_x{\subseteq} \textsf{H}^1$ for $\alpha_x=\gamma_1^{i_1}\ldots \gamma_p^{i_p}\upsilon_1^{j_1}\ldots \upsilon_p^{j_p}$, for some non-negative integers $p,i_1,\ldots, i_p,j_1,\ldots,j_p$, where $\gamma^0=\upsilon^0=\mathbb{I}$, we get that $(\textsf{H}^{1}_{\operatorname{dom}(\varpi\alpha_x)})\varpi\alpha_x{\subseteq} \textsf{V}^1$,
\begin{equation*}
\begin{split}
  (x_\varpi)\mathfrak{I}\cdot(x)\mathfrak{I}
    & =
   \left(e,\bar{1}\right)\cdot\left((\alpha_x)\mathfrak{H}_{\sigma},\bar{0}\right)
= \\
    & =
   \left(e,\bar{1}\right)\cdot\left((\gamma_1^{i_1}\ldots\gamma_p^{i_p}\upsilon_1^{j_1}\ldots \upsilon_p^{j_p})\mathfrak{H}_{\sigma},\bar{0}\right)
= \\
    & =
   \left(e,\bar{1}\right)\cdot\left(a_1^{i_1}\ldots a_p^{i_p}b_1^{j_1}\ldots b_p^{j_p},\bar{0}\right)
= \\
    & =
   \left(e\cdot (a_1^{i_1}\ldots a_p^{i_p}b_1^{j_1}\ldots b_p^{j_p})\mathfrak{f},\bar{1}\cdot \bar{0}\right)
= \\
    & =
   \left((a_1^{i_1}\ldots a_p^{i_p}b_1^{j_1}\ldots b_p^{j_p})\mathfrak{f},\bar{1}\right)
= \\
    & =
   \left(a_1^{j_1}\ldots a_p^{j_p}b_1^{i_1}\ldots b_p^{i_p},\bar{1}\right)
\end{split}
\end{equation*}
and by Proposition~\ref{proposition-2.19},
\begin{equation*}
\begin{split}
  (x_\varpi\cdot x)\mathfrak{I} & =\left((\varpi\alpha_x\varpi)\mathfrak{H}_{\sigma},\bar{1}\right)= \\
    & =\left((\varpi\gamma_1^{i_1}\ldots \gamma_p^{i_p}\upsilon_1^{j_1}\ldots\upsilon_p^{j_p}\varpi)\mathfrak{H}_{\sigma},\bar{1}\right)= \\
    & =\left((\upsilon_1^{i_1}\ldots \upsilon_p^{i_p}\varpi\varpi\gamma_1^{j_1}\ldots\gamma_p^{j_p})\mathfrak{H}_{\sigma},\bar{1}\right)= \\
    & =\left((\upsilon_1^{i_1}\ldots \upsilon_p^{i_p}\gamma_1^{j_1}\ldots\gamma_p^{j_p})\mathfrak{H}_{\sigma},\bar{1}\right)= \\
    & =\left(b_1^{i_1}\ldots b_p^{i_p}a_1^{j_1}\ldots a_p^{j_p},\bar{1}\right)= \\
    & =\left(a_1^{j_1}\ldots a_p^{j_p}b_1^{i_1}\ldots b_p^{i_p},\bar{1}\right);
\end{split}
\end{equation*}

  \item[$(ii)$] in the case when $(\textsf{H}^{1}_{\operatorname{dom}\alpha_x})\alpha_x{\subseteq} \textsf{V}^1$ we get for $\alpha_x=\gamma_1^{i_1}\ldots \gamma_p^{i_p}\upsilon_1^{j_1}\ldots \upsilon_p^{j_p}\varpi$, for some non-negative integers $p,i_1,\ldots, i_p,j_1,\ldots,j_p$, where $\gamma^0=\upsilon^0=\mathbb{I}$, we get that $(\textsf{H}^{1}_{\operatorname{dom}(\varpi\alpha_x\varpi)})\varpi\alpha_x\varpi{\subseteq} \textsf{H}^1$,
\begin{equation*}
\begin{split}
  (x_\varpi)\mathfrak{I}\cdot(x)\mathfrak{I}
    & =
   \left(e,\bar{1}\right)\cdot\left((\alpha_x\varpi)\mathfrak{H}_{\sigma},\bar{1}\right)
= \\
    & =
   \left(e,\bar{1}\right)\cdot\left((\gamma_1^{i_1}\ldots\gamma_p^{i_p}\upsilon_1^{j_1}\ldots \upsilon_p^{j_p}\varpi\varpi)\mathfrak{H}_{\sigma},\bar{1}\right)
= \\
    & =
   \left(e,\bar{1}\right)\cdot\left((\gamma_1^{i_1}\ldots\gamma_p^{i_p}\upsilon_1^{j_1}\ldots \upsilon_p^{j_p})\mathfrak{H}_{\sigma},\bar{1}\right)
= \\
    & =
   \left(e,\bar{1}\right)\cdot\left(a_1^{i_1}\ldots a_p^{i_p}b_1^{j_1}\ldots b_p^{j_p},\bar{1}\right)
= \\
& =
   \left(e\cdot (a_1^{i_1}\ldots a_p^{i_p}b_1^{j_1}\ldots b_p^{j_p})\mathfrak{f},\bar{1}\cdot \bar{1}\right)
= \\
    & =
   \left((a_1^{i_1}\ldots a_p^{i_p}b_1^{j_1}\ldots b_p^{j_p})\mathfrak{f},\bar{0}\right)
= \\
\end{split}
\end{equation*}
\begin{equation*}
\begin{split}
    & =
   \left(a_1^{j_1}\ldots a_p^{j_p}b_1^{i_1}\ldots b_p^{i_p},\bar{0}\right)
\end{split}
\end{equation*}
and by Proposition~\ref{proposition-2.19},
\begin{equation*}
\begin{split}
  (x_\varpi\cdot x)\mathfrak{I} & =\left((\varpi\alpha_x\varpi)\mathfrak{H}_{\sigma},\bar{0}\right)= \\
    & =\left((\varpi\gamma_1^{i_1}\ldots \gamma_p^{i_p}\upsilon_1^{j_1}\ldots\upsilon_p^{j_p}\varpi)\mathfrak{H}_{\sigma},\bar{0}\right)= \\
    & =\left((\upsilon_1^{i_1}\ldots \upsilon_p^{i_p}\varpi\varpi\gamma_1^{j_1}\ldots\gamma_p^{j_p})\mathfrak{H}_{\sigma},\bar{0}\right)= \\
    & =\left((\upsilon_1^{i_1}\ldots \upsilon_p^{i_p}\gamma_1^{j_1}\ldots\gamma_p^{j_p})\mathfrak{H}_{\sigma},\bar{0}\right)= \\
    & =\left(b_1^{i_1}\ldots b_p^{i_p}a_1^{j_1}\ldots a_p^{j_p},\bar{0}\right)= \\
    & =\left(a_1^{j_1}\ldots a_p^{j_p}b_1^{i_1}\ldots b_p^{i_p},\bar{0}\right),
\end{split}
\end{equation*}
\end{itemize}
which implies that $(x_\varpi\cdot x)\mathfrak{I}=(x_\varpi)\mathfrak{I}\cdot(x)\mathfrak{I}$.

Therefore we have showed that $(x_\mathbb{I})\mathfrak{I}$ is the identity element of $\mathscr{P\!O}\!_{\infty}(\mathbb{N}^2_{\leqslant})/\sigma$ and $(x_\varpi)\mathfrak{I}\cdot (x_\varpi)\mathfrak{I}=(x_\mathbb{I})\mathfrak{I}$.

Next we shall show that so defined map $\mathfrak{I}$ is a homomorphism from $\mathscr{P\!O}\!_{\infty}(\mathbb{N}^2_{\leqslant})/\sigma$ into the semigroup $\mathfrak{AM}_\omega\rtimes_{\mathfrak{Q}}\mathbb{Z}_2$. Fix arbitrary elements  $x$ and $y$ of $\mathscr{P\!O}\!_{\infty}(\mathbb{N}^2_{\leqslant})/\sigma$. We consider the following four possible cases:
\begin{itemize}
  \item[$(i)$] $(\textsf{H}^{1}_{\operatorname{dom}\alpha_x})\alpha_x{\subseteq} \textsf{H}^1$ and $(\textsf{H}^{1}_{\operatorname{dom}\alpha_y})\alpha_y{\subseteq} \textsf{H}^1$ for any $\alpha_x,\alpha_y\in\mathscr{P\!O}\!_{\infty}(\mathbb{N}^2_{\leqslant})$ such that $(\alpha_x)\mathfrak{P}_{\sigma}=x$ and $(\alpha_y)\mathfrak{P}_{\sigma}=y$;

  \item[$(ii)$] $(\textsf{H}^{1}_{\operatorname{dom}\alpha_x})\alpha_x{\subseteq} \textsf{V}^1$ and $(\textsf{H}^{1}_{\operatorname{dom}\alpha_y})\alpha_y{\subseteq} \textsf{H}^1$ for any $\alpha_x,\alpha_y\in\mathscr{P\!O}\!_{\infty}(\mathbb{N}^2_{\leqslant})$ such that $(\alpha_x)\mathfrak{P}_{\sigma}=x$ and $(\alpha_y)\mathfrak{P}_{\sigma}=y$;

  \item[$(iii)$] $(\textsf{H}^{1}_{\operatorname{dom}\alpha_x})\alpha_x{\subseteq} \textsf{H}^1$ and $(\textsf{H}^{1}_{\operatorname{dom}\alpha_y})\alpha_y{\subseteq} \textsf{V}^1$ for any $\alpha_x,\alpha_y\in\mathscr{P\!O}\!_{\infty}(\mathbb{N}^2_{\leqslant})$ such that $(\alpha_x)\mathfrak{P}_{\sigma}=x$ and $(\alpha_y)\mathfrak{P}_{\sigma}=y$;

  \item[$(iv)$] $(\textsf{H}^{1}_{\operatorname{dom}\alpha_x})\alpha_x{\subseteq} \textsf{V}^1$ and $(\textsf{H}^{1}_{\operatorname{dom}\alpha_y})\alpha_y{\subseteq} \textsf{V}^1$ for any $\alpha_x,\alpha_y\in\mathscr{P\!O}\!_{\infty}(\mathbb{N}^2_{\leqslant})$ such that $(\alpha_x)\mathfrak{P}_{\sigma}=x$ and $(\alpha_y)\mathfrak{P}_{\sigma}=y$.
\end{itemize}

Assume that $(i)$ hods. Then we have that $\alpha_x,\alpha_y,\alpha_x\alpha_y\in\mathscr{P\!O}\!_{\infty}^{\,+}(\mathbb{N}^2_{\leqslant})$. Since $\sigma$ is a congruence on the semigroup $\mathscr{P\!O}\!_{\infty}(\mathbb{N}^2_{\leqslant})$, we may choose an element $\alpha_{xy}=\alpha_x\alpha_y\in\mathscr{P\!O}\!_{\infty}^{\,+}(\mathbb{N}^2_{\leqslant})$. Then $(\alpha_{xy})\mathfrak{P}_{\sigma}=xy$ Also, since $\mathfrak{P}_{\sigma}\colon \mathscr{P\!O}\!_{\infty}^{\,+}(\mathbb{N}^2_{\leqslant})\rightarrow\mathscr{P\!O}\!_{\infty}(\mathbb{N}^2_{\leqslant})/\sigma$ is the natural homomorphism generated by the congruence $\sigma$ on the semigroup $\mathscr{P\!O}\!_{\infty}(\mathbb{N}^2_{\leqslant})$ we get that
\begin{equation*}
\begin{split}
  (xy)\mathfrak{I} & =((\alpha_{xy})\mathfrak{P}_{\sigma})\mathfrak{I}= \left((\alpha_{xy})\mathfrak{H}_{\sigma},\bar{0}\right)=  \left((\alpha_{x}\alpha_{y})\mathfrak{H}_{\sigma},\bar{0}\right)= \left((\alpha_{x})\mathfrak{H}_{\sigma}\cdot(\alpha_{y})\mathfrak{H}_{\sigma},\bar{0}\cdot\bar{0}\right)= \\
    & = \left((\alpha_{x})\mathfrak{H}_{\sigma},\bar{0}\right)\cdot \left((\alpha_{y})\mathfrak{H}_{\sigma},\bar{0}\right)=  (x)\mathfrak{I}\cdot (y)\mathfrak{I}.
\end{split}
\end{equation*}

If $(ii)$ hods then by Propositions~1 and 3 from \cite{Gutik-Pozdniakova-2016??}, $\alpha_x\varpi,\alpha_y,\alpha_x\alpha_y\varpi\in\mathscr{P\!O}\!_{\infty}^{\,+}(\mathbb{N}^2_{\leqslant})$ and by Lemma~\ref{lemma-2.13} without loss of generality we may assume that
\begin{equation*}
  \alpha_x=\gamma_1^{i_1}\ldots \gamma_p^{i_p}\upsilon_1^{j_1}\ldots \upsilon_p^{j_p}\varpi\qquad \hbox{and} \qquad \alpha_y=\gamma_1^{s_1}\ldots \gamma_p^{s_p}\upsilon_1^{t_1}\ldots \upsilon_p^{t_p},
\end{equation*}
for some non-negative integers $p,i_1,\ldots, i_p,j_1,\ldots,j_p,s_1,\ldots,s_p,t_1,\ldots,t_p$, where $\gamma^0=\upsilon^0=\mathbb{I}$. This and the fact that $\sigma$ is a congruence on the semigroup $\mathscr{P\!O}\!_{\infty}(\mathbb{N}^2_{\leqslant})$, Proposition~\ref{proposition-2.19} imply that
\begin{equation*}
\begin{split}
  (xy)\mathfrak{I} & =\left((\alpha_x\alpha_y\varpi)\mathfrak{H}_{\sigma},\bar{1}\right)= \\
    & =\left((\gamma_1^{i_1}\ldots \gamma_p^{i_p}\upsilon_1^{j_1}\ldots \upsilon_p^{j_p}\varpi\gamma_1^{s_1}\ldots \gamma_p^{s_p}\upsilon_1^{t_1}\ldots \upsilon_p^{t_p}\varpi)\mathfrak{H}_{\sigma},\bar{1}\right)= \\
    & = \left((\gamma_1^{i_1}\ldots \gamma_p^{i_p}\upsilon_1^{j_1}\ldots \upsilon_p^{j_p}\upsilon_1^{s_1}\ldots \upsilon_p^{s_p}\varpi\varpi\gamma_1^{t_1}\ldots \gamma_p^{t_p})\mathfrak{H}_{\sigma},\bar{1}\right)=\\
    & = \left((\gamma_1^{i_1}\ldots \gamma_p^{i_p}\upsilon_1^{j_1}\ldots \upsilon_p^{j_p}\upsilon_1^{s_1}\ldots \upsilon_p^{s_p}\gamma_1^{t_1}\ldots \gamma_p^{t_p})\mathfrak{H}_{\sigma},\bar{1}\right)=\\
    & = \left(a_1^{i_1}\ldots a_p^{i_p}b_1^{j_1}\ldots b_p^{j_p}b_1^{s_1}\ldots b_p^{s_p}a_1^{t_1}\ldots a_p^{t_p},\bar{1}\right)=\\
    & =\left(a_1^{i_1}\ldots a_p^{i_p}b_1^{j_1}\ldots b_p^{j_p}(a_1^{s_1}\ldots a_p^{s_p}b_1^{t_1}\ldots b_p^{t_p})\mathfrak{f},\bar{1}\cdot\bar{0}\right)=\\
    & =\left(a_1^{i_1}\ldots a_p^{i_p}b_1^{j_1}\ldots b_p^{j_p},\bar{1}\right) \cdot
   \left(a_1^{s_1}\ldots a_p^{s_p}b_1^{t_1}\ldots b_p^{t_p},\bar{0}\right)= \\
    & =\left((\gamma_1^{i_1}\ldots \gamma_p^{i_p}\upsilon_1^{j_1}\ldots \upsilon_p^{j_p})\mathfrak{H}_{\sigma},\bar{1}\right) \cdot
   \left((\gamma_1^{s_1}\ldots \gamma_p^{s_p}\upsilon_1^{t_1}\ldots \upsilon_p^{t_p})\mathfrak{H}_{\sigma},\bar{0}\right)= \\
    & =\left((\gamma_1^{i_1}\ldots \gamma_p^{i_p}\upsilon_1^{j_1}\ldots \upsilon_p^{j_p}\varpi\varpi)\mathfrak{H}_{\sigma},\bar{1}\right) \cdot
   \left((\gamma_1^{s_1}\ldots \gamma_p^{s_p}\upsilon_1^{t_1}\ldots \upsilon_p^{t_p})\mathfrak{H}_{\sigma},\bar{0}\right)=\\
    & =\left((\alpha_x\varpi)\mathfrak{H}_{\sigma},\bar{1}\right) \cdot
   \left((\alpha_y)\mathfrak{H}_{\sigma},\bar{0}\right)=\\
    & = (x)\mathfrak{I}\cdot (y)\mathfrak{I}.
\end{split}
\end{equation*}

If $(iii)$ hods then by Propositions~1 and 3 from \cite{Gutik-Pozdniakova-2016??}, $\alpha_x,\alpha_y\varpi,\alpha_x\alpha_y\varpi\in\mathscr{P\!O}\!_{\infty}^{\,+}(\mathbb{N}^2_{\leqslant})$ and by Lemma~\ref{lemma-2.13} without loss of generality we may assume that
\begin{equation*}
  \alpha_x=\gamma_1^{i_1}\ldots \gamma_p^{i_p}\upsilon_1^{j_1}\ldots \upsilon_p^{j_p}\qquad \hbox{and} \qquad \alpha_y=\gamma_1^{s_1}\ldots \gamma_p^{s_p}\upsilon_1^{t_1}\ldots \upsilon_p^{t_p}\varpi,
\end{equation*}
for some non-negative integers $p,i_1,\ldots, i_p,j_1,\ldots,j_p,s_1,\ldots,s_p,t_1,\ldots,t_p$, where $\gamma^0=\upsilon^0=\mathbb{I}$. Since $\sigma$ is a congruence on the semigroup $\mathscr{P\!O}\!_{\infty}(\mathbb{N}^2_{\leqslant})$, this and Proposition~\ref{proposition-2.19} imply that

\begin{equation*}
\begin{split}
  (xy)\mathfrak{I} & =\left((\alpha_x\alpha_y\varpi)\mathfrak{H}_{\sigma},\bar{1}\right)= \\
    & =\left((\gamma_1^{i_1}\ldots \gamma_p^{i_p}\upsilon_1^{j_1}\ldots \upsilon_p^{j_p}\gamma_1^{s_1}\ldots \gamma_p^{s_p}\upsilon_1^{t_1}\ldots \upsilon_p^{t_p}\varpi\varpi)\mathfrak{H}_{\sigma},\bar{1}\right)= \\
    & = \left((\gamma_1^{i_1}\ldots \gamma_p^{i_p}\upsilon_1^{j_1}\ldots \upsilon_p^{j_p}\gamma_1^{s_1}\ldots \gamma_p^{s_p}\upsilon_1^{t_1}\ldots \upsilon_p^{t_p})\mathfrak{H}_{\sigma},\bar{1}\right)=\\
    & = \left(a_1^{i_1}\ldots a_p^{i_p}b_1^{j_1}\ldots b_p^{j_p}a_1^{s_1}\ldots a_p^{s_p}b_1^{t_1}\ldots b_p^{t_p}, \bar{0}\cdot\bar{1}\right)=\\
    & =\left(a_1^{i_1}\ldots a_p^{i_p}b_1^{j_1}\ldots b_p^{j_p},\bar{0}\right) \cdot
   \left(a_1^{s_1}\ldots a_p^{s_p}b_1^{t_1}\ldots b_p^{t_p},\bar{1}\right)= \\
    & =\left((\gamma_1^{i_1}\ldots \gamma_p^{i_p}\upsilon_1^{j_1}\ldots \upsilon_p^{j_p})\mathfrak{H}_{\sigma},\bar{0}\right) \cdot
   \left((\gamma_1^{s_1}\ldots \gamma_p^{s_p}\upsilon_1^{t_1}\ldots \upsilon_p^{t_p})\mathfrak{H}_{\sigma},\bar{1}\right)= \\
    & =\left((\gamma_1^{i_1}\ldots \gamma_p^{i_p}\upsilon_1^{j_1}\ldots \upsilon_p^{j_p})\mathfrak{H}_{\sigma},\bar{0}\right) \cdot
   \left((\gamma_1^{s_1}\ldots \gamma_p^{s_p}\upsilon_1^{t_1}\ldots \upsilon_p^{t_p}\varpi\varpi)\mathfrak{H}_{\sigma},\bar{1}\right)= \\
    & =\left((\alpha_x)\mathfrak{H}_{\sigma},\bar{0}\right) \cdot
   \left((\alpha_y\varpi)\mathfrak{H}_{\sigma},\bar{1}\right)=\\
    & = (x)\mathfrak{I}\cdot (y)\mathfrak{I}.
\end{split}
\end{equation*}

Assume that $(iv)$ hods. Then by Propositions~1 and 3 from \cite{Gutik-Pozdniakova-2016??} we have that $\alpha_x\varpi,\alpha_y\varpi,\alpha_x\alpha_y$, $\alpha_x\varpi\alpha_y\varpi\in\mathscr{P\!O}\!_{\infty}^{\,+}(\mathbb{N}^2_{\leqslant})$ and by Lemma~\ref{lemma-2.13} without loss of generality we may assume that
\begin{equation*}
  \alpha_x=\gamma_1^{i_1}\ldots \gamma_p^{i_p}\upsilon_1^{j_1}\ldots \upsilon_p^{j_p}\varpi\qquad \hbox{and} \qquad \alpha_y=\gamma_1^{s_1}\ldots \gamma_p^{s_p}\upsilon_1^{t_1}\ldots \upsilon_p^{t_p}\varpi,
\end{equation*}
for some non-negative integers $p,i_1,\ldots, i_p,j_1,\ldots,j_p,s_1,\ldots,s_p,t_1,\ldots,t_p$, where $\gamma^0=\upsilon^0=\mathbb{I}$. Since $\sigma$ is a congruence on the semigroup $\mathscr{P\!O}\!_{\infty}(\mathbb{N}^2_{\leqslant})$, this and  Proposition~\ref{proposition-2.19} imply that
\begin{equation*}
\begin{split}
  (xy)\mathfrak{I} & =\left((\alpha_x\alpha_y)\mathfrak{H}_{\sigma},\bar{0}\right)= \\
    & =\left((\gamma_1^{i_1}\ldots \gamma_p^{i_p}\upsilon_1^{j_1}\ldots \upsilon_p^{j_p}\varpi\gamma_1^{s_1}\ldots \gamma_p^{s_p}\upsilon_1^{t_1}\ldots \upsilon_p^{t_p}\varpi)\mathfrak{H}_{\sigma},\bar{0}\right)= \\
    & =\left((\gamma_1^{i_1}\ldots \gamma_p^{i_p}\upsilon_1^{j_1}\ldots \upsilon_p^{j_p}\upsilon_1^{s_1}\ldots \upsilon_p^{s_p}\varpi\varpi\gamma_1^{t_1}\ldots \gamma_p^{t_p})\mathfrak{H}_{\sigma},\bar{0}\right)= \\
    & =\left((\gamma_1^{i_1}\ldots \gamma_p^{i_p}\upsilon_1^{j_1}\ldots \upsilon_p^{j_p}\upsilon_1^{s_1}\ldots \upsilon_p^{s_p}\gamma_1^{t_1}\ldots \gamma_p^{t_p})\mathfrak{H}_{\sigma},\bar{0}\right)=  \\
    & =\left(a_1^{i_1}\ldots a_p^{i_p}b_1^{j_1}\ldots b_p^{j_p}b_1^{s_1}\ldots b_p^{s_p}a_1^{t_1}\ldots a_p^{t_p},\bar{0}\right)=  \\
    & =\left(a_1^{i_1}\ldots a_p^{i_p}b_1^{j_1}\ldots b_p^{j_p}\cdot(a_1^{s_1}\ldots a_p^{s_p}b_1^{t_1}\ldots b_p^{t_p})\mathfrak{f},\bar{1}\cdot\bar{1}\right) \\
    & =\left(a_1^{i_1}\ldots a_p^{i_p}b_1^{j_1}\ldots b_p^{j_p},\bar{1}\right) \cdot
   \left(a_1^{s_1}\ldots a_p^{s_p}b_1^{t_1}\ldots b_p^{t_p},\bar{1}\right)= \\
    & =\left((\gamma_1^{i_1}\ldots \gamma_p^{i_p}\upsilon_1^{j_1}\ldots \upsilon_p^{j_p})\mathfrak{H}_{\sigma},\bar{1}\right) \cdot
   \left((\gamma_1^{s_1}\ldots \gamma_p^{s_p}\upsilon_1^{t_1}\ldots \upsilon_p^{t_p})\mathfrak{H}_{\sigma},\bar{1}\right)= \\
    & =\left((\gamma_1^{i_1}\ldots \gamma_p^{i_p}\upsilon_1^{j_1}\ldots \upsilon_p^{j_p}\varpi\varpi)\mathfrak{H}_{\sigma},\bar{1}\right) \cdot
   \left((\gamma_1^{s_1}\ldots \gamma_p^{s_p}\upsilon_1^{t_1}\ldots \upsilon_p^{t_p}\varpi\varpi)\mathfrak{H}_{\sigma},\bar{1}\right)= \\
    & =\left((\alpha_x\varpi)\mathfrak{H}_{\sigma},\bar{1}\right) \cdot
   \left((\alpha_y\varpi)\mathfrak{H}_{\sigma},\bar{1}\right)=\\
    & = (x)\mathfrak{I}\cdot (y)\mathfrak{I}.
\end{split}
\end{equation*}

Thus the map $\mathfrak{I}\colon \mathscr{P\!O}\!_{\infty}(\mathbb{N}^2_{\leqslant})/\sigma \to \mathfrak{AM}_\omega\rtimes_{\mathfrak{Q}} \mathbb{Z}_2$ is a homomorphism. Also, since $(x_\mathbb{I})\mathfrak{I}=\left(e,\bar{0}\right)$, $(x_\varpi)\mathfrak{I}=\left(e,\bar{1}\right)$ and for any $\alpha_x=\gamma_1^{i_1}\ldots \gamma_p^{i_p}\upsilon_1^{j_1}\ldots \upsilon_p^{j_p}$, where $p,i_1,\ldots, i_p,j_1,\ldots,j_p$ are some positive integers, our above arguments imply that
\begin{equation*}
  (x)\mathfrak{I}=\left(a_1^{i_1}\ldots a_p^{i_p}b_1^{j_1},\bar{0}\right) \qquad \hbox{and} \qquad (y)\mathfrak{I}=\left(a_1^{i_1}\ldots a_p^{i_p}b_1^{j_1},\bar{1}\right),
\end{equation*}
where $x=(\alpha_x)\mathfrak{P}_{\sigma}$ and  $y=(\alpha_x\varpi)\mathfrak{P}_{\sigma}$. This implies that the homomorphism $\mathfrak{I}$ is surjective.

Now suppose that $(x)\mathfrak{I}=(y)\mathfrak{I}=(u,g)$ for some $x,y\in \mathscr{P\!O}\!_{\infty}(\mathbb{N}^2_{\leqslant})/\sigma$. Then there exist $\alpha_x,\alpha_y\in \mathscr{P\!O}\!_{\infty}^{\,+}(\mathbb{N}^2_{\leqslant})$ such that $(\alpha_x)\mathfrak{P}_{\sigma}=x$ and $(\alpha_y)\mathfrak{P}_{\sigma}=y$ in the case when $g=\bar{0}$, and $(\alpha_x\varpi)\mathfrak{P}_{\sigma}=x$ and  $(\alpha_y\varpi)\mathfrak{P}_{\sigma}=y$ in the case when $g=\bar{1}$. If $g=\bar{0}$ then $x,y\in\mathscr{P\!O}\!_{\infty}^{\,+}(\mathbb{N}^2_{\leqslant})$ and the condition $\alpha_x\sigma\alpha_y$ in $\mathscr{P\!O}\!_{\infty}^{\,+}(\mathbb{N}^2_{\leqslant})$ implies the equality $x=y$. Similarly, if $g=\bar{1}$ then $x,y\in\mathscr{P\!O}\!_{\infty}(\mathbb{N}^2_{\leqslant})\setminus\mathscr{P\!O}\!_{\infty}^{\,+}(\mathbb{N}^2_{\leqslant})$ and the condition $\alpha_x\varpi\sigma\alpha_y\varpi$ in $\mathscr{P\!O}\!_{\infty}^{\,+}(\mathbb{N}^2_{\leqslant})$ implies the equality $x=y$.
Hence $\mathfrak{I}\colon \mathscr{P\!O}\!_{\infty}(\mathbb{N}^2_{\leqslant})/\sigma \to \mathfrak{AM}_\omega\rtimes_{\mathfrak{Q}} \mathbb{Z}_2$ is an isomorphism.
\end{proof}


\section*{Acknowledgements}
The authors acknowledge Taras Banakh and Alex Ravsky for their comments and suggestions.

\end{document}